\documentclass[pdflatex,sn-mathphys-ay]{sn-jnl}

\usepackage{graphicx}%
\usepackage{multirow}%
\usepackage{amsmath,amssymb,amsfonts}%
\usepackage{amsthm}%
\usepackage{mathrsfs}%
\usepackage[title]{appendix}%
\usepackage{xcolor}%
\usepackage{textcomp}%
\usepackage{manyfoot}%
\usepackage{booktabs}%
\usepackage{algorithm}%
\usepackage{algorithmicx}%
\usepackage{algpseudocode}%
\usepackage{listings}%
\usepackage{mathtools}
\usepackage{mathrsfs}
\usepackage{natbib}

\theoremstyle{thmstyleone}%
\newtheorem{theorem}{Theorem}[section]
   
\newtheorem{lemma}[theorem]{Lemma}

\theoremstyle{definition}

\newtheorem{example}[theorem]{Example}

\theoremstyle{thmstyletwo}%

\theoremstyle{thmstylethree}%

\newcommand{\E}{\mathbf{E}}
\renewcommand{\P}{\mathbf{P}}
\newcommand{\Prob}[1]{\P\left\{#1\right\}}

\newcommand{\N}{\mathbb{N}}
\newcommand{\NN}{\mathbb{N}}
\newcommand{\R}{\mathbb{R}}
\newcommand{\ZZ}{\mathbb{Z}}

\newcommand{\bN}{\mathbf{N}}

\newcommand{\bW}{\mathbb{W}}
\newcommand{\bM}{\mathbb{M}}
\newcommand{\bY}{\mathbb{Y}}

\newcommand{\sA}{\mathscr{A}}
\newcommand{\sD}{\mathscr{D}}

\newcommand{\sM}{\mathcal{M}}
\newcommand{\sN}{\mathcal{N}}
\newcommand{\sS}{\mathcal{S}}
\newcommand{\sY}{\mathcal{Y}}

\newcommand{\sW}{\mathcal{W}}

\newcommand{\me}{\mathsf{m}}
\newcommand{\one}{\mathbf{1}}

\renewcommand{\epsilon}{\varepsilon}
\newcommand{\eps}{\varepsilon}
\newcommand{\nT}{\lfloor nT \rfloor }
\newcommand{\dm}{\vartheta}

\newcommand{\dGH}{\mathsf{d}_{\mathrm{GH}}}

\newcommand{\dH}[1][p]{\dm_{\mathrm{H},#1}}
\newcommand{\rhoH}{\rho_{\mathrm{H},p}}
\newcommand{\rhoHI}{\rho_{\mathrm{H},\infty}}

\renewcommand{\d}{\,\mathrm d}

\newcommand{\dsim}{\stackrel{d}{\sim}}
\newcommand{\dto}{\stackrel{d}{\to}}
\newcommand{\vto}{\xrightarrow{v}}

\newcommand{\vstop}[1][p]{\xrightarrow{\sS_{#1}}}

\numberwithin{equation}{section}

\begin{document}

\title[Ranges of Extremal Processes and Heavy-Tailed Random
Walks]{Ranges of Extremal Processes and Heavy-Tailed 
      Random Walks in Spaces of Growing Dimension}

\author*[1]{\fnm{Bochen} \sur{Jin}}\email{bochen.jin@unibe.ch}

\author[1]{\fnm{Ilya} \sur{Molchanov}}\email{ilya.molchanov@unibe.ch}
\equalcont{These authors contributed equally to this work.}

\affil[1]{\orgdiv{Institute of Mathematical Statistics and Actuarial
    Science}, \orgname{University of Bern}, \orgaddress{\street{Alpeneggstrasse 22}, \city{Bern}, \postcode{3012}, \state{Bern}, \country{Switzerland}}}

\abstract{ We consider extremal processes and random walks generated
  by heavy-tailed random vectors taking values in ${\mathbb R}^d$ endowed with
  the $\ell_p$ metric. We establish limit theorems for the ranges
    of associated paths in the triangular array setting when both the
  number of steps $n$ and the dimension $d$ grow to infinity.  It is
  shown that it is possible to transform the paths by suitable
  linear isometric embeddings of $\ell_p$ such that the
  transformed paths converge in distribution and to identify the limit
  in terms of a Poisson cluster process.  
  These results also imply the
  convergence in distribution of the paths viewed as finite metric
  spaces in the space of metric spaces equipped with the
  Gromov--Hausdorff metric.  Furthermore, we prove convergence in
  distribution of the transformed paths in the space of counting
  measures on the line equipped with a Hausdorff metric induced by a
  suitable $\ell_p$-type distance between counting measures.}



\maketitle

\section{Introduction}
\label{sec:introduction}

Let $X_n^{(d)}$, $n\geq1$, be a sequence of i.i.d.\ random vectors in
$\R^d_+=[0,\infty)^d$. Under regular variation of their tails,
the sequence of normalised partial coordinatewise maxima processes,
\begin{equation}
  \label{eq:M-seq}
  M_n^{(d)}=X_1^{(d)}\vee\cdots\vee X_n^{(d)}, \quad n\geq1,
\end{equation}
converges as $n\to\infty$ to an extremal process, see
\cite{res87}. The main aim of this paper is to study convergence of
this partial maxima process (and its variant for the sums with the
summands taking values in the whole $\R^d$) when both the number of
observations $n$ and the dimension $d$ increase to infinity.

In the case of sums, this problem has been considered by Kabluchko and
Marynych in \cite{MR4828862} and by Kabluchko, Marynych and Raschel
in \cite{MR4774174}. Their results deal
with the random walk
\begin{displaymath}
  S_k^{(d)}=X_1^{(d)}+\cdots+X_k^{(d)},\quad k\geq 0,
\end{displaymath}
in $\R^d$. The values of this random walk are viewed as a finite
metric space
\begin{equation}
  \label{eq:W-seq}
  \bW_n^{(d)}=\big\{0,S_1^{(d)},\dots,S_n^{(d)}\big\},
\end{equation}
which is embedded in $\R^d$ with the $\ell_2$-metric. Under square
integrability and several further assumptions listed in
\cite{MR4828862}, the metric space $\bW_n^{(d)}$ normalised by
$\sqrt{n}$ converges in probability under the Gromov--Hausdorff metric
to the Wiener spiral. This space is the set of indicators
$\mathbf{1}_{[0,t]}$, $t\in[0,1]$, embedded in $L^2\big([0,1]\big)$;
it is isometric to the interval $[0,1]$ with the metric
$r(t,s)=|t-s|^{1/2}$. The Gromov--Hausdorff metric is defined as the
infimum of the Hausdorff distances between images of the spaces under
isometric maps into an arbitrarily chosen third metric space, see
\cite{MR1835418}.

In the heavy-tailed case and under a suitable normalisation, the
authors of \cite{MR4774174} proved that $\bW_n^{(d)}$ converges to a
random metric space derived from a sequence-space version of a
subordinator (called a crinkled subordinator). All mentioned results
impose some orthogonality type conditions on the increments of the
random walk and require the uniform smallness of each component of the
increments.  In the heavy-tailed case, it is also assumed that the
Euclidean norm of the summand is regularly varying (in the triangular
arrays scheme, since the distribution of a summand changes with $d$).

As indicated in the beginning, we study the extremal version (the
\emph{maximum scheme}) of the setting of \cite{MR4774174}, where the
random walk (the partial sum process)
is replaced by a sequence of successive maxima (the partial maxima
process), and the
underlying metric space is
\begin{equation}
  \label{eq:sMn}
  \bM_n^{(d)}=\big\{0,M_1^{(d)},\dots,M_n^{(d)}\big\}
\end{equation}
considered as a subset of $\R^d_+$ with the $\ell_p$-distance for any
$p\in[1,\infty]$. Moreover, we also consider the \emph{random walk
  scheme} with heavy-tailed summands, replacing the $\ell_2$-setting
of \cite{MR4774174} with a general $\ell_p$-distance on $\R^d$. 

In the Hilbert setting, the authors of \cite{MR4828862,MR4774174}
first proved convergence in the Gromov--Hausdorff metric and then
obtained convergence in the Hausdorff metric up to ambient isometries.
Outside that setting, abstract metric isometries must be distinguished
from linear isometric embeddings of the ambient sequence space. For
example, in $\ell_1$ the sets $K=\{0,e_1,e_2\}$ and
$L=\{0,e_1,-e_1\}$ are isometric as metric spaces. However, their
Hausdorff distance after arbitrary linear isometric embeddings into
$\ell_1$ is at least one. Indeed, a distance smaller than one would
match the two nonzero images of $K$ to opposite unit vectors. Their
sum would then have norm smaller than two, whereas linear isometry
preserves $\|e_1+e_2\|_1=2$. Convergence to the Wiener spiral in the
$\ell_p$ setting with $p\neq2$ was established in \cite{MR5098150}.

The studies of random metric spaces rely on identifying the spaces
up to metric isometries. The linear isometric embeddings used here
have additional structure. For $1\leq p<\infty$, $p\neq2$, they are
described by disjoint supports of the images of the coordinate vectors,
as specified in Theorem~\ref{th:isometries}; see \cite{MR105017} and
\cite[Theorem~7.4.1]{MR4249569}. For $p=\infty$, we use an explicit
coordinate construction rather than a classification of all embeddings.
This motivates retaining the coordinate configuration of an increment,
not just its norm, while treating vectors of increasing dimension in a
permutation-invariant way.
For this, we associate each
vector with a finite counting measure on $\R\setminus\{0\}$ having
atoms at its nonzero coordinates. Dealing with vectors of
varying (increasing) dimension requires working with triangular arrays
of counting measures. On the technical side, this calls for tools
suitable for handling the regular variation property in the triangular
array scheme.  In view of this, we have advanced in the studies of
extremal distributions in spaces of growing dimension. A recent work
in this direction due to Lederer and Oesting \cite{led:oes23} mostly
restricts consideration to the multivariate H\"usler--Reiss
distribution.

Furthermore, leaving the $\ell_2$ framework renders it impossible to
use standard expressions for the norm of the sum of vectors. This
problem is circumvented by truncating small components,
approximating random vectors by sparse vectors containing only their
large components, and imposing a condition which essentially
means that the coordinate positions of large components overlap with
negligible probability.

We obtain distributional convergence results of three different
types.
First, we map the sets $\bM_n^{(d)}$ and $\bW_n^{(d)}$ into the
space $\bN_p$ of counting measures on $\R\setminus\{0\}$ and show that their
normalised images converge in the Hausdorff metric induced by an
$\ell_p$-type distance between counting measures.
Second, we identify the limits of $\bM_n^{(d)}$ and
$\bW_n^{(d)}$ considered as random compact subsets of $\ell_p$ in the
Hausdorff metric on $\ell_p$ up to isometries. This means that the
isometric images of these sets converge in the Hausdorff metric
induced by the standard $\ell_p$-metric on the space of
sequences. This latter result immediately implies the third one,
saying that the metric spaces $\big(\bM_n^{(d)},\|\cdot\|_p\big)$ and
$\big(\bW_n^{(d)},\|\cdot\|_p\big)$ converge in distribution in the
Gromov--Hausdorff metric.

The limit is derived from a Poisson cluster process with the cluster
distribution determined by the tail measure of a single increment
$X_1^{(d)}$.  Both $\bM_n^{(d)}$ and $\bW_n^{(d)}$ converge to the
same limit when their respective hypotheses hold;  this is explained
by the fact that the dominant summands 
contribute on asymptotically disjoint sets of coordinates, so that
their maximum is asymptotically equivalent to the sum.

The paper is organised as follows.  Section~\ref{sec:main_result}
introduces the main concepts of metrics on the space of counting
measures and the regular variation in the triangular array scheme. The
key object is the tail measure $\nu$ arising as the vague limit of the
rescaled law of the counting measures generated by the components of
random vectors of growing dimension. This section concludes with the
main result, Theorem~\ref{th:single-jump}. Its proof is presented in
Section~\ref{sec:proof-theorem-ref} followed by examples in
Section~\ref{sec:examples}.

We consider the basic examples of vectors
with i.i.d.\ components and vectors with only one nonzero
component, which were also included in \cite{MR4774174}. Another
example arises from vectors derived from finite-dimensional
distributions of moving maxima time series.
Section~\ref{sec:example:-long-range} discusses at length another
example which arises from time series with long-range dependence. In
this example we identify the cluster structure of a long-range
dependent time series which follows the logistic distribution. We show
that the cluster distribution is derived from a multivariate Sibuya
distribution. This example does not appear to be directly covered by
the techniques developed in \cite{basrak18}, specifically devoted to
the study of clusters in time series.

The appendices describe a permutation-invariant metric on the space of
sequences, state the result on isometries of $\ell_p$, and present
several useful results concerning vague convergence of measures on the
space of counting measures with $\ell_p$ metrics.

\section{Extremal processes and random walks of growing dimension} 
\label{sec:main_result}

\subsection{Metrics on the space of counting measures}
\label{sec:metr-space-count}

For $p\in[1,\infty)$, let $\ell_p$ be the space of $p$-summable
sequences $x=(x_1,x_2,\ldots)$ with the norm
$\|x\|_p=\big(\sum_{i=1}^\infty |x_i|^p)^{1/p}$ and the metric
$\rho_p(x,y)=\|x-y\|_p$. The space $\ell_\infty$ consists of bounded
sequences with the norm $\|x\|_\infty=\sup_i |x_i|$.  Denote by
$\rhoH$ the Hausdorff distance between subsets of $\ell_p$, namely,
\begin{displaymath}
  \rhoH(K,L)
  =\max\Big(\sup_{x\in K}\inf_{y\in L} \|x-y\|_p,
  \sup_{y\in L}\inf_{x\in K} \|x-y\|_p\Big)
\end{displaymath}
for $K,L\subset\ell_p$.

Let $\R_* = \R\setminus\{0\}$, and let $\bN_\infty$ be the space of
counting measures on $\R_*$ which are finite on $\{x:|x|>s\}$ for
every $s>0$. The empty measure is denoted by $0$. Put
\begin{displaymath}
  \|\mu\|_\infty=\sup\{|x|:x\text{ is an atom of }\mu\},
  \qquad \|0\|_\infty=0.
\end{displaymath}
This quantity is finite for every $\mu\in\bN_\infty$. For $1\leq p<\infty$,
we work on the smaller state space
\begin{displaymath}
  \bN_p=\left\{\mu\in\bN_\infty:
    \|\mu\|_p^p=\int_{\R_*}|x|^p\mu(\d x)<\infty\right\}.
\end{displaymath}
Atoms at zero are always discarded. Thus, $\delta_z$ in a formula for
a counting measure means the restriction of the usual Dirac measure to
$\R_*$, and it is the empty measure when $z=0$.

Distances between counting measures obtained by matching their atoms
were considered in \cite{MR1190904,MR2454027,MR2271178}. Here the
matching allows arbitrary zero padding: an unmatched atom may be
matched to zero, which is not counted as an atom of the measure.
Let $u_j(\mu)$ be the positive atoms of $\mu$ in nonincreasing order,
and let $v_j(\mu)$ be the absolute values of its negative atoms in
nonincreasing order, each sequence padded with zeros if necessary.
The optimal matching distance is
\begin{equation}
  \label{eq:matching-metric}
  \dm_p(\mu,\mu')=
  \begin{cases}
  \displaystyle\bigg(\sum_{j\geq1}|u_j(\mu)-u_j(\mu')|^p+
    \sum_{j\geq1}|v_j(\mu)-v_j(\mu')|^p\bigg)^{1/p},&1\leq p<\infty,\\[1ex]
  \displaystyle\max\bigg\{\sup_j|u_j(\mu)-u_j(\mu')|,
    \sup_j|v_j(\mu)-v_j(\mu')|\bigg\},&p=\infty.
  \end{cases}
\end{equation}
Indeed, matching opposite signs can be replaced by matching each to
zero without increasing the cost, and the optimal matching within
each sign is ordered. Appendix~\ref{sec:a:-perm-invar} gives the
corresponding permutation formulation. In particular,
$\dm_p(\mu,0)=\|\mu\|_p$.

The ranked pairs $(u(\mu),v(\mu))$ (where $u$ and $v$ are composed of
their coordinates $u_j$ and $v_j$) form a closed subset of
$\ell_p\times\ell_p$ for $p<\infty$, with the product $p$-norm, and
of $c_0\times c_0$ for $p=\infty$, with the maximum norm. Here $c_0$
is the space of sequences converging to zero. Consequently,
$(\bN_p,\dm_p)$ is a complete separable metric space. We use its
Borel $\sigma$-algebra $\sN_p$; equivalently, this is generated by
the evaluation maps $\mu\mapsto\mu(A)$ for Borel $A\subset\R_*$.
The Hausdorff distance between compact subsets of $\bN_p$ is denoted
by $\dH$.

For $\mu_n,\mu\in\bN_\infty$, write $\mu_n\vto\mu$ and say that
$\mu_n$ \emph{vaguely converges} to $\mu$ if
$\int f\d\mu_n\to\int f\d\mu$ for every bounded continuous
$f:\R\to\R$ which vanishes on a neighbourhood of zero. Equivalently,
for each $s>0$ such that $\mu(\{-s,s\})=0$, the atoms in $\{|x|>s\}$
can eventually be matched one-to-one, with all matched locations
converging. This includes agreement of their finite numbers and
excludes additional atoms escaping to infinity; see
\cite[Proposition~2.8]{bas:plan19}. Applying this statement to a
countable sequence of such levels decreasing to zero gives consistent
convergence for all the possibly infinitely many nonzero atoms of
$\mu$.

For $p=\infty$, vague convergence alone is equivalent to
convergence in $\dm_\infty$. Indeed, at every continuity level
$s>0$, the finitely many atoms with absolute value greater than
$s$ can eventually be matched with errors tending to zero,
while every remaining atom can be matched to a padding zero
with error at most $s$. Letting $s\downarrow0$ proves the claim.
In particular, convergence of $\|\mu_n\|_\infty$ is automatic.

Since $\dm_\infty\leq\dm_p$, convergence in $\dm_p$ implies vague
convergence. For finite $p$ it is strictly stronger: the measures
$\mu_n=n\delta_{n^{-1/p}}$ converge vaguely to $0$, but
$\|\mu_n\|_p=1$. We will also use the following consequence of the
ranked representation:
\begin{equation}
  \label{eq:topology-norm}
  \dm_p(\mu_n,\mu)\to0
  \quad\text{if and only if}\quad
  \mu_n\vto\mu\ \text{and}\ \|\mu_n\|_p\to\|\mu\|_p,
  \qquad p<\infty.
\end{equation}
For the reverse implication, convergence of any finite number of
ranked coordinates and convergence of the total $p$th powers of the
norms make the remaining $p$-norm tails uniformly small.

A counting measure is scaled by scaling its atoms: $t\mu=\sum_i
\delta_{t x_i}$ if $\mu=\sum_i\delta_{x_i}$ and $t>0$. Scaling is
continuous in $\dm_p$.

We also need vague convergence of measures on the counting-measure
space. As in \cite{bas:plan19} and \cite[Section~2.3.1]{bmm25},
let $\sS_p$ consist of all subsets of $\bN_p$ bounded away from the
empty measure: $\sA\in\sS_p$ if $\inf_{\mu\in\sA}\|\mu\|_p>0$.
Tail measures are defined on $\bN_p\setminus\{0\}$, and are extended
by zero at $0$. Measures $\nu_n,\nu$ finite on each measurable set
in $\sS_p$ satisfy $\nu_n\vstop\nu$ (and are said to converge vaguely
on $\sS_p$) if
\begin{equation}
  \label{eq:vsto}
  \int g\d\nu_n\to\int g\d\nu
\end{equation}
for every bounded $\dm_p$-continuous function $g$ with support in
some member of $\sS_p$. Define
\begin{equation}
  \label{eq:ov-nu}
  \overline\nu_p((s,\infty))
  =\nu(\{\mu\in\bN_p:\|\mu\|_p>s\}),\qquad s>0.
\end{equation}
When comparing finite-$p$ convergence with $p=\infty$, a measure on
$\bN_p\setminus\{0\}$ is also viewed as a measure on
$\bN_\infty\setminus\{0\}$ by the natural inclusion.

Define the coordinate-to-measure map by
\begin{equation}
  \label{eq:chi}
  \me(x_1,\ldots,x_d)=\sum_{i:x_i\neq0}\delta_{x_i}.
\end{equation}
It extends to $\ell_p$ for $p<\infty$ and to $c_0$ for $p=\infty$,
and satisfies $\|\me(x)\|_p=\|x\|_p$ and
$\dm_p(\me(x),\me(y))\leq\|x-y\|_p$. The pseudometric
$\rho_{\Pi,p}$ of Appendix~\ref{sec:a:-perm-invar} identifies two
sequences precisely when they have the same nonzero coordinate
multiset. Thus, its metric quotient is $(\bN_p,\dm_p)$; the quotient
is by zero pseudodistance, not merely by finite permutations.
For $x\in\R^d$ write $|x|=(|x_1|,\ldots,|x_d|)$.

Throughout, an \emph{ambient isometry} means a linear isometric
embedding $U:\ell_p\to\ell_p$, not necessarily surjective.
An isometry between abstract metric spaces only means a
distance-preserving bijection. The Hausdorff comparison of compact
sets $K,L\subset\ell_p$ up to ambient isometries is
\begin{displaymath}
  \Delta_p(K,L)=\inf_{U,U'}\rhoH(U(K),U'(L)),
\end{displaymath}
where the infimum is over linear isometric embeddings. We state the
probabilistic convergence in the explicit coupling form in
Theorem~\ref{th:single-jump}(II). For $p=\infty$, all the random sets
considered here lie in $c_0$, and the embeddings constructed in the
proof preserve $c_0$ and are explicitly defined on all of $\ell_\infty$.

\subsection{Main results}
\label{sec:main-results}

Let
\begin{displaymath}
  X_i^{(d)}=(X_{i1}^{(d)},\dots, X_{id}^{(d)}), \quad i\in\NN,
\end{displaymath}
be a sequence of independent and identically distributed random vectors in
$\R^{d}$, which are independent copies of a random vector $X^{(d)}$.
Recall the sequence of consecutive componentwise maxima from
\eqref{eq:M-seq} and the corresponding finite set $\bM_n^{(d)}$ from
\eqref{eq:sMn}, which are constructed, assuming that the random vectors
have nonnegative components. Furthermore, recall that $\bW_n^{(d)}$ is
the set of values of the random walk with increments given by
$X_1^{(d)},\dots, X_n^{(d)}$, see \eqref{eq:W-seq}. We fix a
$p\in[1,\infty]$ and endow $\bM_n^{(d)}$ and $\bW_n^{(d)}$ with the
$\ell_p$-metric as a subspace of $\R^d$.

With each random vector $X^{(d)}$ in $\R^d$ we associate the random
counting measure defined as
\begin{displaymath}
  \chi^{(d)}=\me(X^{(d)}).
\end{displaymath}
The random counting measures $\chi_i^{(d)}=\me(X_i^{(d)})$, $i\in\NN$,
derived from i.i.d.\ copies of $X^{(d)}$, are i.i.d.\ copies of
$\chi^{(d)}$.

Let $a_{n,d}$ be a sequence depending on $n$ and $d$, where $d=d(n)$
depends on $n$ so that $d(n)\to\infty$ as $n\to\infty$. We usually
suppress the dependence on $d$ and write $a_{n}=a_{n,d(n)}$. 
The first condition concerns the vague convergence of the normalised
counting measures generated by the components of $X^{(d)}$.

\begin{enumerate}
\item[(A)] We have
  \begin{equation}
    \label{eq:nu}
    n \Prob{a_n^{-1} \chi^{(d)} \in \cdot}
    \vstop \nu(\cdot)\quad \text{as}\; n\to\infty
  \end{equation}
  with some $p\in[1,\infty]$ and a nontrivial measure $\nu$ on the
  space $\bN_p$, which is necessarily finite on sets from $\sS_p$.  If
  $p\in[1,\infty)$, assume also that
  \begin{equation}
    \label{eq:levy-integral}
    \int \min(1,\|\mu\|_p^p)\nu(\!\d\mu)<\infty.
  \end{equation}
\end{enumerate}

In the following we call $\nu$ the \emph{tail measure} of
$\chi^{(d)}$. If $\nu$ is supported by finite counting measures, it is
called regular; otherwise, singular. If the dimension $d$ is fixed,
condition \eqref{eq:nu} becomes the regular variation property of
$\chi^{(d)}$. With varying $d$, \eqref{eq:nu} is a kind of regular
variation property formulated for a triangular array scheme, see
\cite[Remark~3.19]{bmm25}. Unlike the case of a fixed 
dimension, where it is possible to infer from \eqref{eq:nu} that $\nu$
is homogeneous, such a conclusion does not follow in our setting,
since $\chi^{(d)}$ changes with $d$. Furthermore, for 
\eqref{eq:nu} to hold, a particular growth regime of $d(n)$ in
relation to $n$ may be necessary.
Note that maxima and sums of random vectors cannot be directly related
to operations on counting measures, because the order of components of
a random vector is lost when passing to the corresponding counting
measure using \eqref{eq:chi}. Therefore, it is not possible to
directly apply \eqref{eq:nu} in order to obtain a limit theorem for
the partial maxima or the sums of random vectors.

Fortunately, when the significant coordinate supports are disjoint,
the $p$th power of the norm of a sum, or of a nonnegative maximum,
is the sum of the individual $p$th powers for finite $p$; for
$p=\infty$ the corresponding norm is the maximum of the individual
norms. We obtain this geometry asymptotically by requiring that the
large components of the independent
random vectors $X_i^{(d)}$ and $X_j^{(d)}$ do not overlap
asymptotically for $i\neq j$.
This fact is a property of the sequence
$a_n^{-1}(|X_1^{(d)}|\wedge |X_2^{(d)}|)$, $n\geq1$, obtained by
taking the coordinatewise minima of the absolute values for two
independent copies of $X^{(d)}$ and normalised by the same sequence
$(a_n)_{n\in\N}$ as in \eqref{eq:nu}.  This brings us to the following
condition, which replaces the asymptotic orthogonality condition from
\cite{MR4774174}.

\begin{enumerate}
\item[(B)] For every $\eps>0$ and every $s>0$ such that
  $\overline\nu_\infty((s,\infty))>0$, and for the sequence $(a_n)$
  from (A), as $n\to\infty$,
  \begin{equation}
    \Prob{\big\||X_1^{(d)}|\wedge |X_2^{(d)}|\big\|_\infty >\eps
      a_n\,\Big|\, \|X_1^{(d)}\|_\infty>sa_n, \|X_2^{(d)}\|_\infty>sa_n}
    \to 0.
  \end{equation}
\end{enumerate}

Conditional probabilities in (B) are taken to be zero when the
conditioning event has probability zero.
Under (A), condition (B), imposed for every $\eps>0$
and every $s>0$ such that
$\overline\nu_\infty((s,\infty))>0$, is equivalent to
\begin{equation}
  \label{equ_con_4}
  n^2\Prob{\big\||X_1^{(d)}|\wedge |X_2^{(d)}|\big\|_\infty >\eps a_n}
  \to 0 \quad \text{as} \; n \to \infty.
\end{equation}
To see the equivalence, write
$c(t)=\overline\nu_\infty((t,\infty))$,
$q_n(t)=\Prob{\|X^{(d)}\|_\infty>ta_n}$, and let $O_{n,\eps}$
denote the overlap event in \eqref{equ_con_4}. Nontriviality of
$\nu$ implies $c(t)>0$ for all sufficiently small $t>0$.
Choose a continuity level $0<t<\eps$ with $c(t)>0$. The overlap
event forces both sup norms to exceed $ta_n$, so independence gives
\[
 n^2\Prob{O_{n,\eps}}
 =[nq_n(t)]^2
 \Prob{O_{n,\eps}\mid
   \|X_1^{(d)}\|_\infty>ta_n,\ \|X_2^{(d)}\|_\infty>ta_n}.
\]
Since $nq_n(t)\to c(t)<\infty$ by
Lemma~\ref{lemma:infty-convergence}, (B) implies
\eqref{equ_con_4}. Conversely, for a level $s$ with $c(s)>0$,
choose a continuity level $u>s$ with $c(u)>0$.
Then $\liminf_n nq_n(s)\geq c(u)>0$, and
\[
 \Prob{O_{n,\eps}\mid
   \|X_1^{(d)}\|_\infty>sa_n,\ \|X_2^{(d)}\|_\infty>sa_n}
 \leq\frac{n^2\Prob{O_{n,\eps}}}{[nq_n(s)]^2}\to0.
\]
This proves the equivalence for every $p\in[1,\infty]$.

\begin{lemma}
  \label{lemma:exchangeable-overlap}
  Suppose that (A) holds for some $p\in[1,\infty]$, that
  $d=d(n)\to\infty$, and that the components of $X^{(d)}$ are
  exchangeable for every $d$. Then the unconditional overlap
  condition \eqref{equ_con_4} holds. Consequently, the conditional
  probability in (B) tends to zero for every $\eps>0$ and every
  $s>0$ such that $\overline\nu_\infty((s,\infty))>0$.
\end{lemma}
\begin{proof}
  Since convergence in $\dm_p$ implies convergence in $\dm_\infty$ by
  Lemma~\ref{lemma:lp-vague-total}, we 
  assume that (A) holds with $p=\infty$.
  Fix $\eps>0$, and write
  $\nu_n(\cdot)=n\P\{a_n^{-1}\chi^{(d)}\in\cdot\}$,
  $C_n=\chi^{(d)}(\{x:|x|>\eps a_n\})$, and $q_n=\Prob{C_n>0}$.
  We first show that
  \begin{equation}
    \label{eq:exchangeable-count-tightness}
    \limsup_{n\to\infty}nq_n<\infty,
    \qquad
    \lim_{L\to\infty}\limsup_{n\to\infty}
    n\P\{C_n>L\}=0.
  \end{equation}
  Choose a continuous function $\psi_\eps:\R\to[0,1]$ that vanishes on
  $[-\eps/2,\eps/2]$ and equals one when $|x|\geq\eps$. Put
  \[
    F_\eps(\mu)=\int\psi_\eps(x)\,\mu(\d x),
    \qquad
    g_L(\mu)=1\wedge\frac{F_\eps(\mu)}{L},\quad L\geq1.
  \]
  The function $F_\eps$ is finite and continuous in $\dm_\infty$, by
  the definition of vague convergence of counting measures.  Each
  $g_L$ is bounded, continuous, and vanishes whenever
  $\|\mu\|_p\leq\eps/2$. Hence (A) gives
  \[
    \int g_L\d\nu_n\to \int g_L\d\nu.
  \]
  Since $C_n\leq F_\eps(a_n^{-1}\chi^{(d)})$,
  \[
    nq_n\leq\int g_1\d\nu_n,
    \qquad
    n\Prob{C_n>L}\leq\int g_L\d\nu_n.
  \]
  Moreover, $g_L(\mu)\downarrow0$ as $L\to\infty$, and
  \[
    0\leq g_L(\mu)
    \leq\one_{\{\|\mu\|_\infty>\eps/2\}},
    \qquad
    \nu(\{\mu:\|\mu\|_\infty>\eps/2\})<\infty.
  \]
  Dominated convergence proves
  \eqref{eq:exchangeable-count-tightness}.
  
  For the independent copies $X_1^{(d)}$ and $X_2^{(d)}$, set
  \[
    A_n^{(j)}=\{i\leq d:|X_{ji}^{(d)}|>\eps a_n\},
    \qquad C_n^{(j)}=|A_n^{(j)}|,\quad j=1,2.
  \]
  Conditional on $C_n^{(j)}=k$, exchangeability makes
  $A_n^{(j)}$ a uniform $k$-element subset of $\{1,\ldots,d\}$.
  The two sets remain independent after conditioning on their
  respective sizes. Thus, whenever the conditioning event has
  positive probability,
  \begin{multline*}
    \Prob{A_n^{(1)}\cap A_n^{(2)}\ne\emptyset
      \mid C_n^{(1)}=k,\ C_n^{(2)}=\ell}\\
    \leq
      \E\big[|A_n^{(1)}\cap A_n^{(2)}|
      \mid C_n^{(1)}=k,\ C_n^{(2)}=\ell\big]
      =\frac{k\ell}{d}.
  \end{multline*}
  Then
  \[
    \{A_n^{(1)}\cap A_n^{(2)}\ne\varnothing\}
    =\{\||X_1^{(d)}|\wedge|X_2^{(d)}|\|_\infty>\eps a_n\}.
  \]
  For each integer $L\geq1$, the preceding bound and independence
  give
  \[
    \Prob{A_n^{(1)}\cap A_n^{(2)}\ne\emptyset}
    \leq\frac{L^2}{d}q_n^2+2\Prob{C_n>L}q_n.
  \]
  Indeed, when both counts are at most $L$, the collision
  probability is at most $L^2/d$; if either count exceeds $L$,
  an overlap still requires the other count to be positive.
  Consequently,
  \[
    n^2\Prob{A_n^{(1)}\cap A_n^{(2)}\ne\emptyset}
    \leq\frac{L^2}{d}(nq_n)^2+2\big[n\P\{C_n>L\}\big](nq_n).
  \]
  First let $n\to\infty$ with $L$ fixed, and then let
  $L\to\infty$. Equation~\eqref{eq:exchangeable-count-tightness} and
  $d(n)\to\infty$ yield \eqref{equ_con_4}.
  The equivalence just proved gives the asserted conditional
  convergence at every level $s$ with
  $\overline\nu_\infty((s,\infty))>0$.
\end{proof}

The next condition addresses the integrability property of random
vectors.

\begin{enumerate}
\item[(C)] In the maximum scheme with $p\in[1,\infty)$,
  \begin{equation}
    \label{eq:b1}
    \lim_{s \downarrow 0} \limsup_{n \to \infty}
    \frac{n}{a_n^p}\E\Big[\big\|X_1^{(d)}\big\|_p^p
    \one_{\|X_1^{(d)}\|_p\leq sa_n}\Big]=0,
  \end{equation}
  and in the random walk setting with $p\in[1,\infty]$,
  \begin{equation}
    \label{eq:b}
    \lim_{s \downarrow 0} \limsup_{n \to \infty}
    \frac{n}{a_n}\E\Big[\big\|X_1^{(d)}\big\|_p
    \one_{\|X_1^{(d)}\|_p\leq sa_n}\Big]=0.
  \end{equation}
\end{enumerate}

Condition \eqref{eq:b} implies \eqref{eq:b1}, since
\begin{displaymath}
  a_n^{-p}\E\Big[\|X_1^{(d)}\|_p^p\one_{\|X_1^{(d)}\|_p\leq
    sa_n}\Big]
  \leq s^{p-1} a_n^{-1}\E\Big[\|X_1^{(d)}\big\|_p
    \one_{\|X_1^{(d)}\|_p\leq sa_n}\Big]. 
\end{displaymath}

Now we describe the limiting objects. Let
\begin{displaymath}
  \eta=\sum_k\delta_{(t_k,\mu_k)}
\end{displaymath}
be a Poisson process on $\R_+\times(\bN_p\setminus\{0\})$ with
intensity $\lambda\otimes\nu$, where $\lambda$ is Lebesgue measure.
Set $\upsilon_0=0$ and
\begin{displaymath}
  \upsilon_t=\sum_{k:t_k\leq t}\mu_k,\qquad t>0.
\end{displaymath}
For a fixed $T>0$, enumerate all the points with $t_k\leq T$ in a
measurable way and put
\begin{equation}
  \label{eq:Y}
  Y_t=\sum_{k:t_k\leq T}e_k\|\mu_k\|_p\one_{t_k\leq t},
  \qquad 0\leq t\leq T.
\end{equation}
Note that the isometry class of the limiting object does not depend on
the chosen enumeration.
Here $(e_k)$ are the canonical coordinate vectors.
Condition (A)
ensures that $\upsilon_t\in\bN_p$ and $Y_t\in\ell_p$, with
$Y_t\in c_0$ when $p=\infty$, and
$\|Y_t\|_p=\|\upsilon_t\|_p$. For $p=2$, $Y$ is the crinkled
subordinator of \cite{MR4774174}.

The two limiting random compact sets are the \emph{closed} ranges
\begin{displaymath}
  \sY_T=\operatorname{cl}_{\dm_p}\{\upsilon_t:0\leq t\leq T\},
  \qquad
  \bY_T=\operatorname{cl}_{\ell_p}\{Y_t:0\leq t\leq T\}.
\end{displaymath}
Indeed, retaining only jumps with $\|\mu_k\|_p>s$ gives finite-step
processes. The uniform error is at most
\begin{displaymath}
  \Big(\sum_{t_k\leq T,\,\|\mu_k\|_p\leq s}\|\mu_k\|_p^p\Big)^{1/p}
\end{displaymath}
for finite $p$, and at most $s$ for $p=\infty$. The former tends to
zero almost surely by \eqref{eq:levy-integral}. Thus, both processes
are right-continuous with left limits in their complete state spaces,
and their closed ranges are compact. In particular,
\begin{displaymath}
  \bY_T=\{Y_t:0\leq t\leq T\}\cup\{Y_{t-}:0<t\leq T\},
\end{displaymath}
with the analogous identity for $\sY_T$.

For $0\leq t\leq t'\leq T$, define
\begin{displaymath}
  \rho_{\nu,p}(t,t')=
  \begin{cases}
  \displaystyle\left(\sum_{t<t_k\leq t'}\|\mu_k\|_p^p\right)^{1/p},
    &1\leq p<\infty,\\[1ex]
  \displaystyle\sup_{t<t_k\leq t'}\|\mu_k\|_\infty\; ,&p=\infty,
  \end{cases}
\end{displaymath}
and extend it symmetrically, with the supremum of an empty set equal
to zero.
This is generally a pseudometric, not a metric.
Since $\|Y_{t'}-Y_t\|_p=\rho_{\nu,p}(t,t')$, the compact metric space
$(\bY_T,\|\cdot\|_p)$ is isometric to the completion of the quotient
of $([0,T],\rho_{\nu,p})$ by zero pseudodistance. The quotient handles
intervals without jumps, including finite-activity tail measures, and
the completion adds any unattained left limits. The primary metric
limit below is the closed range $\bY_T$. Its law up to isometry only
depends on $\overline\nu_p$ from \eqref{eq:ov-nu}.

Our results are of three closely related kinds. First, we consider the
convergence of the images of $\bM_{\nT}^{(d)}$ and $\bW_{\nT}^{(d)}$
under the map $\me$. These images $\sM_{\nT}^{(d)}$ and $\sW_{\nT}^{(d)}$
are compact subsets of $\bN_p$ and we prove their convergence in
distribution with respect to the Hausdorff metric. Second, we consider
convergence of metric spaces obtained by considering $\bM_{\nT}^{(d)}$
and $\bW_{\nT}^{(d)}$ as subsets of $\ell_p$ equipped
with the Hausdorff metric up to isometries.
Finally, we infer from this the convergence in the
Gromov--Hausdorff sense.

\begin{theorem}
  \label{th:single-jump}
  Assume (A) and (B) for some $p\in[1,\infty]$ and a nonzero tail
    measure $\nu$ on $\bN_p\setminus\{0\}$.  For the maximum scheme,
    assume nonnegative components and, when $p<\infty$, condition
    \eqref{eq:b1} of (C) holds. For the random-walk scheme, assume
    \eqref{eq:b}. Then, for each $T>0$, conclusions (I)--(III) hold
    for each scheme whose hypotheses are satisfied.
  \begin{enumerate}
  \item[(I)] The random sets $a_n^{-1}\sM_{\nT}^{(d)}$ and
    $a_n^{-1}\sW_{\nT}^{(d)}$ converge in distribution as $n\to\infty$
    (with respect to the Hausdorff distance $\dH$ on $\bN_p$) to $\sY_T$.
  \item[(II)] The random sets $a_n^{-1}\bM_{\nT}^{(d)}$ and
    $a_n^{-1}\bW_{\nT}^{(d)}$ converge in distribution up to linear
    isometric embeddings to $\bY_T$. More precisely, for either
    scheme, writing $K_n$ for the corresponding unnormalised random
    set, there are coupled copies of $(K_n)_{n\geq1}$ and $\bY_T$,
    and random linear isometric embeddings $U_n:\ell_p\to\ell_p$,
    such that
    \begin{equation}
      \label{eq:embedded-coupling}
      \rhoH(a_n^{-1}K_n,U_n(\bY_T))\to 0
      \quad\text{in probability}.
    \end{equation}
  \item[(III)] The random metric spaces
    $\big(a_n^{-1}\bM_{\nT}^{(d)},\|\cdot\|_p\big)$ and
    $\big(a_n^{-1}\bW_{\nT}^{(d)},\|\cdot\|_p\big)$ converge in
    distribution as $n\to\infty$ (with respect to the Gromov--Hausdorff
    distance) to $\big(\bY_T,\|\cdot\|_p\big)$.
  \end{enumerate}
\end{theorem}

The key effect implied by the growing dimension setting is that in
\eqref{eq:Y} the different $\mu_k$ contribute additively to the metric
of the limiting space, even in the maximum scheme. If the dimension
does not grow, then this is no longer the case and the ordering of the
atoms of $\mu_k$ influences the limit. For instance, in the maximum
scheme with $p=\infty$, the limiting object is the path of the
extremal process generated by a point process on $\R_+\times\R^d$, see
\cite{res87}.

Note that the statements (I) and (II) of Theorem~\ref{th:single-jump}
imply the convergence of the corresponding metric spaces in the
Gromov--Hausdorff metric, e.g.,
$\big(a_n^{-1}\sM_{\nT}^{(d)},\dm_p\big)$ converges to $(\sY_T,\dm_p)$
and $\big(a_n^{-1}\bM_{\nT}^{(d)},\|\cdot\|_p\big)$ converges to
$(\bY_T,\|\cdot\|_p)$.

If $p=2$, then the limiting metric space $(\bY_T,\|\cdot\|_p)$
coincides with the limit obtained in \cite{MR4774174}, namely, (II)
and (III) hold, while (I) does not have an analogue in
\cite{MR4774174}. In this case, it is possible to work with the
$\ell_2$-norms of the vectors and counting measures, including those
appearing in the limit, and so represent the limiting space as a
subset of the space of square summable sequences. Condition (A)
implies a regular variation type property of the $\ell_2$ norms of the
increments, and the
random walk variant \eqref{eq:b} of (C) controls the cumulative
contribution of increments with small norm. In general,
Condition (B) is not directly comparable with the orthogonality
condition imposed in \cite{MR4774174}.

\section{Proof of \protect{Theorem~\ref{th:single-jump}}}
\label{sec:proof-theorem-ref}

The proof relies on approximating the sets $\sM_{\nT}^{(d)}$ and
$\sW_{\nT}^{(d)}$ by retaining only those increments whose
$\ell_p$-norms exceed a fixed threshold.  Fix an $s>0$, and let
\begin{equation*}
  X_{k}^{(d)}(s)= X_k^{(d)} \one_{\|X_k^{(d)}\|_p>a_ns},
  \quad k \geq 1.
\end{equation*}
These vectors yield the sets
\begin{align*}
  \bM_{\nT}^{(d)}(s)
  &=\Big\{0,\bigvee_{i=1}^j X_{i}^{(d)}(s),\; j=1,\dots,\nT\Big\},\\
  \bW_{\nT}^{(d)}(s)
  &=\Big\{0,\sum_{i=1}^j X_{i}^{(d)}(s),\; j=1,\dots,\nT\Big\},
\end{align*}
and their images $\sM_{\nT}^{(d)}(s)$ and $\sW_{\nT}^{(d)}(s)$ under
the map $\me$, which are used as approximations of
$\sM_{\nT}^{(d)}$ and $\sW_{\nT}^{(d)}$.
The limiting set is approximated using the same parameter $s>0$ by
letting $\sY_T^{(s)}$ be the (finite, hence closed) range of $\upsilon_t^{(s)}$,
$0\leq t\leq T$, with
\begin{displaymath}
  \upsilon_t^{(s)}=\sum_{k: t_k\leq t,\|\mu_k\|_p>s} \mu_k,\quad 
  0\leq t\leq T. 
\end{displaymath}
For each $T>0$, almost surely there are finitely many points
$\big(t_{k_i}, \mu_{k_i}\big)$, $i=1,\dots,N_T(s)$, of $\eta$ with
$\|\mu_{k_i}\|_p>s$ and $t_{k_i}\leq T$.  This number is a Poisson
random variable $N_T=N_T(s)$ with mean $T\overline{\nu}_p((s,\infty))$, see
\eqref{eq:ov-nu}.  Without loss of generality assume that
$t_{k_1}\leq t_{k_2}\leq\cdots\leq t_{k_{N_T}}$.
Furthermore, denote
\begin{displaymath}
  Y_t^{(s)}=\sum_{k:t_k\leq T}
  e_k\|\mu_k\|_p\one_{t_k\leq t}\one_{\|\mu_k\|_p>s},
  \quad 0\leq t\leq T,
\end{displaymath}
and let $\bY_T^{(s)}=\mathrm{cl}\{Y_t^{(s)},0\leq t\leq T\}$.

Choose a deterministic sequence $s_m\downarrow0$ satisfying
$\nu(\{\mu:\|\mu\|_p=s_m\})=0$. Such a sequence exists because there
are at most countably many exceptional levels. The following three
lemmas and \cite[Theorem~3.2]{MR1700749} imply (I). For (II), use the
common coupling of Lemma~\ref{le:3} at these countably many levels
and choose $s=s(n)\downarrow0$ by a diagonal argument. The embeddings
constructed there preserve the truncation errors in
Lemmas~\ref{le:1}--\ref{le:2}, yielding
\eqref{eq:embedded-coupling}. Finally, (III) follows from (II) and
\begin{displaymath}
  \dGH(a_n^{-1}K_n,\bY_T)
  \leq\rhoH(a_n^{-1}K_n,U_n(\bY_T)).
\end{displaymath}

First, we address the convergence of approximations of the limiting
sets. 

\begin{lemma}
  \label{le:1}
  The set $\sY_T^{(s)}$ converges in probability to $\sY_T$, and the
  set $\bY_T^{(s)}$ converges in probability to $\bY_T$ (in the
  Hausdorff metric on $\bN_p$ and $\ell_p$, respectively) as
  $s \downarrow 0$.
\end{lemma}
\begin{proof}
  If $p=\infty$, the statement follows from the inequalities
  \begin{align*}
    \dH[\infty]\big(\sY_T^{(s)}, \sY_T\big)
    \leq \sup_{0\leq t\leq T}
      \dm_\infty(\upsilon_t^{(s)},\upsilon_t)
    \leq \sup_{k: t_k\leq T, \|\mu_k\|_\infty\leq s} \|\mu_k\|_\infty
      \leq s
  \end{align*}
  and $\rhoHI(\bY_T^{(s)},\bY_T)\leq s$ with the convention that the
  supremum over the empty set equals 0.  If $p\in[1,\infty)$, then the
  inequality $\dm_p(\mu+\mu',\mu)\leq \|\mu'\|_p$ yields that
  \begin{align*}
    \dH\big(\sY_T^{(s)}, \sY_T\big)
    \leq \dm_p\Big(\sum_{0\leq t_k\leq T, \|\mu_k\|_p\leq s}\mu_k,0\Big)
    =\Big(\sum_{k: t_k\leq T} \|\mu_k\|_p^p\one_{\|\mu_k\|_p\leq
    s}\Big)^{1/p}
  \end{align*}
  and
  \begin{displaymath}
    \rhoH(\bY_T^{(s)},\bY_T)\leq
    \Big(\sum_{k:t_k\leq T} \|\mu_k\|_p^p \one_{\|\mu_k\|_p\leq s}\Big)^{1/p}. 
  \end{displaymath}
  By the Markov inequality, the right-hand side
  converges in probability to zero, since by \eqref{eq:levy-integral}
  \begin{displaymath}
    \E \sum_{k: t_k\leq T}\|\mu_k\|_p^p
    \one_{\|\mu_k\|_p\leq s}
    = T \int_{(0,s]} x^p\overline\nu_p(\d x)\to 0
    \quad \text{as}\; s\downarrow 0. \qedhere
  \end{displaymath}
\end{proof}

The next result provides a uniform bound on the Hausdorff distance
between the approximating sets. Recall that the Hausdorff metric on
$\bN_p$ is denoted by $\dH$ and on $\ell_p$ by $\rhoH$ and that in the
maximum scheme all vectors are assumed to have nonnegative
components.

\begin{lemma}
  \label{le:2}
  Assume that (C) holds (apart from the maximum scheme with
  $p=\infty$, where it is not required). Then, for all $\eps>0$, 
  \begin{align*}
    \lim_{s \downarrow 0} \limsup_{n \to \infty}
    \Prob{\dH \Big( a_n^{-1}\sM_{\nT}^{(d)}(s),
      a_n^{-1}\sM_{\nT}^{(d)} \Big) > \eps} = 0,\\
    \lim_{s \downarrow 0} \limsup_{n \to \infty}
    \Prob{\dH \Big( a_n^{-1}\sW_{\nT}^{(d)}(s),
      a_n^{-1}\sW_{\nT}^{(d)} \Big) > \eps} = 0,
  \end{align*}
  and 
  \begin{align}
    \label{equ_ful_3}  
    \lim_{s \downarrow 0} \limsup_{n \to \infty}
    \Prob{\rhoH\Big( a_n^{-1}\bM_{\nT}^{(d)}(s),
    a_n^{-1}\bM_{\nT}^{(d)} \Big) > \eps} = 0,\\
    \label{equ_ful_3Z}   
    \lim_{s \downarrow 0} \limsup_{n \to \infty}
    \Prob{\rhoH \Big( a_n^{-1}\bW_{\nT}^{(d)}(s),
      a_n^{-1}\bW_{\nT}^{(d)} \Big) > \eps} = 0.
  \end{align}
\end{lemma}
\begin{proof}
  Since $\dm_p(\me(x),\me(y))\leq \|x-y\|_p$, the corresponding
  Hausdorff distances satisfy the same inequality, and therefore it
  suffices to prove the second set of statements.
  Write
  \[
    R_{i,n}(s)=a_n^{-1}X_i^{(d)}
    \one_{\{\|X_i^{(d)}\|_p\leq sa_n\}}.
  \]
  For nonnegative maxima and $p<\infty$, coordinatewise comparison
  gives
  \[
    \rho_{H,p}\big(a_n^{-1}\mathbb M_{\lfloor nT\rfloor}^{(d)},
    a_n^{-1}\mathbb M_{\lfloor nT\rfloor}^{(d)}(s)\big)^p
    \leq\sum_{i=1}^{\lfloor nT\rfloor}\|R_{i,n}(s)\|_p^p.
  \]
  For $p=\infty$ this Hausdorff distance is at most $s$. For random
  walks and every $p\in[1,\infty]$, the triangle inequality gives
  \[
    \rho_{H,p}\big(a_n^{-1}\mathbb W_{\lfloor nT\rfloor}^{(d)},
    a_n^{-1}\mathbb W_{\lfloor nT\rfloor}^{(d)}(s)\big)
    \leq\sum_{i=1}^{\lfloor nT\rfloor}\|R_{i,n}(s)\|_p.
  \]
  Markov's inequality and the respective part of (C) prove the ambient
  statements.
\end{proof}

\begin{lemma}
  \label{le:3}
  Assume (A) and (B). For every $s>0$ such that
  $\nu(\{\mu:\|\mu\|_p=s\})=0$, the sets
  $a_n^{-1}\sM_{\nT}^{(d)}(s)$ and $a_n^{-1}\sW_{\nT}^{(d)}(s)$
  converge in distribution to $\sY_T^{(s)}$ in $\dH$.
  The corresponding sets $a_n^{-1}\bM_{\nT}^{(d)}(s)$ and
  $a_n^{-1}\bW_{\nT}^{(d)}(s)$ converge to $\bY_T^{(s)}$ up to
  linear isometric embeddings, in the coupling sense of
  \eqref{eq:embedded-coupling}. The maximum scheme uses nonnegative
  increments.
\end{lemma}
\begin{proof}
  By Lemma~\ref{lemma:pp} and \cite[Proposition~3.21]{res87},
  \begin{displaymath}
    \eta_n=\sum_{k\geq1:\,\chi_k^{(d)}\neq0}
      \delta_{(k/n,a_n^{-1}\chi_k^{(d)})}
    \dto\eta=\sum_k\delta_{(t_k,\mu_k)}.
  \end{displaymath}
  The point-process topology tests bounded time intervals and marks
  bounded away from $0$ in $\bN_p$. Its state space is Polish.
  By a strong representation, as in \cite{MR3266500}, we may couple
  the processes so that this convergence holds almost surely.
  The coupling can be lifted to distributional copies of the full
  vector arrays using their regular conditional distributions given
  $\eta_n$. We keep the same notation for these copies.

  The boundary of $[0,T]\times\{\mu:\|\mu\|_p>s\}$ has zero
  intensity: its time boundary has Lebesgue measure zero, and its
  mark boundary is null by the choice of $s$. Hence, almost
  surely the restrictions of $\eta_n$ and $\eta$ eventually have the
  same finite number $N_T=N_T(s)$ of atoms. Enumerate the limiting
  atoms as $(t_{k_j},\mu_{k_j})$, $1\leq j\leq N_T$, in strictly
  increasing time order. The matched atoms of $\eta_n$ then have the
  same order for all sufficiently large $n$, and, as $n\to\infty$, 
  \begin{equation}
    \label{eq:result_1}
    k_j(n)/n\to t_{k_j},\qquad
    \dm_p(a_n^{-1}\chi_{k_j(n)}^{(d)},\mu_{k_j})\to0,
    \quad 1\leq j\leq N_T.
  \end{equation}
  The empty case $N_T=0$ is understood throughout, with all sums zero.

  Choose deterministic levels $\gamma\downarrow0$ for which
  $\nu(\{\mu:\|\mu\|_p>s,\ \mu(\{-\gamma,\gamma\})>0\})=0$.
  Such a sequence exists: under the finite measure
  $\nu|_{\{\|\mu\|_p>s\}}$, each mark has only finitely many atoms
  outside every neighbourhood of zero. Put
  $V_{nj}=a_n^{-1}X_{k_j(n)}^{(d)}$ and
  \begin{displaymath}
    I_{nj}=\big\{i\in\{1,\dots,d\}:|V_{nj,i}|>\gamma\big\},\qquad
    V_{nj}^\gamma=\big(V_{nj,i}\one_{i\in I_{nj}}\big)_{i\geq1}.
  \end{displaymath}
  Let $\mu_{k_j}^\gamma$ be the restriction of $\mu_{k_j}$ to
  $\{|x|>\gamma\}$. The absence of atoms at $\pm\gamma$ and
  \eqref{eq:result_1} imply
  \begin{align}
    \label{eq:6_bis}
    \dm_p(\me(V_{nj}^\gamma),\mu_{k_j}^\gamma)&\to0,\\
    \label{eq:6_new}
    \|V_{nj}-V_{nj}^\gamma\|_p
      &\to\|\mu_{k_j}-\mu_{k_j}^\gamma\|_p \quad \text{as}\;
        n\to\infty. 
  \end{align}
  To see the second assertion, use the ranked matching in
  \eqref{eq:matching-metric}. Only finitely many coordinates can be
  near $\pm\gamma$, and none of the limiting coordinates equals
  either boundary point. The matching therefore respects the
  restriction for all sufficiently large $n$. The complementary
  restrictions converge in $\dm_p$ as well, and their norms converge
  by the reverse triangle inequality. The same argument applies in
  the supremum norm.

  Let $\sD_n$ be the event that the coordinate supports above
  $\gamma a_n$ of all rows with $\|X_i^{(d)}\|_p>sa_n$,
  $i\leq\nT$, are pairwise disjoint. Once the matching has stabilised,
  this is the disjointness of the $I_{nj}$.
  Fix a continuity level $b\in(0,\gamma)$ for the pushforward
  of $\nu$ under $\|\cdot\|_\infty$, chosen small enough that
  $\overline\nu_\infty((b,\infty))>0$, and put
  $q_n=\Prob{\|X_1^{(d)}\|_\infty>ba_n}$.
  By (A) and
  Lemma~\ref{lemma:infty-convergence}, $nq_n$ has a finite limit.
  Any coordinate overlap above $\gamma a_n$ forces both supremum
  norms above $b a_n$. By independence of the rows, writing $c_n$ for
  the conditional probability in (B) with 
  $\eps=\gamma$ and conditioning level $b$, we obtain
  \begin{displaymath}
    \P(\sD_n^c)
      \leq n^2T^2\Prob{\big\||X_1^{(d)}|\wedge|X_2^{(d)}|\big\|_\infty
                  >\gamma a_n}
      =T^2(nq_n)^2c_n\to 0.
  \end{displaymath}
  This bound uses all retained rows before matching, and therefore
  holds for every $n$.

  For the matched rows define
  \begin{align*}
    R_n(\gamma)&=\sum_{j=1}^{N_T}\|V_{nj}-V_{nj}^\gamma\|_p,\qquad 
    R(\gamma)=\sum_{j=1}^{N_T}\|\mu_{k_j}-\mu_{k_j}^\gamma\|_p,\\
    E_n(\gamma)&=\sum_{j=1}^{N_T}
          \dm_p(\me(V_{nj}^\gamma),\mu_{k_j}^\gamma).
  \end{align*}
  For fixed $\gamma$, \eqref{eq:6_bis}--\eqref{eq:6_new} give
  $R_n(\gamma)\to R(\gamma)$ and $E_n(\gamma)\to0$ almost surely.
  Since $N_T<\infty$ and every mark belongs to $\bN_p$,
  $R(\gamma)\downarrow0$ almost surely as $\gamma\downarrow0$.
  The triangle inequality and concatenation of atom matchings show
  that, on $\sD_n$ and after the matching has stabilised,
  \begin{displaymath}
    \dH(a_n^{-1}\sW_{\nT}^{(d)}(s),\sY_T^{(s)})
       \leq R_n(\gamma)+E_n(\gamma)+R(\gamma).
  \end{displaymath}
  Indeed, truncating each increment costs at most its norm remainder,
  disjoint supports turn the counting measure of a sum into the sum
  of counting measures, and the last two terms respectively match
  the retained atoms and restore the discarded limiting atoms. These
  bounds hold simultaneously for every partial sum in time order.
  For nonnegative vectors on disjoint supports the sum equals the
  coordinatewise maximum; the same estimate follows for maxima,
  since their truncation error is bounded by the sum of the individual
  norm remainders.

  For completeness, we give the ambient embedding and its error
  explicitly.
  Let $\mathscr H_n(\gamma)$ be the event that the finite
  restrictions admit the above increasing-time matching and every
  matched truncated vector $V_{nj}^{\gamma}$ is nonzero.  By eventual
  matching and the choice of $\gamma$,
  \[
    \limsup_{n\to\infty}\Prob{\mathscr H_n(\gamma)^c}
    \leq\Prob{\text{some }j\leq N_T:
      \|\mu_{k_j}\|_\infty\leq\gamma}
    \to0\qquad(\gamma\downarrow0).
  \]
  On $\mathscr D_n\cap\mathscr H_n(\gamma)$, write
  $U=U_{n,\gamma}$ and set
  \[
    U(e_{k_j})=V_{nj}^{\gamma}/\|V_{nj}^{\gamma}\|_p,
    \qquad 1\leq j\leq N_T.
  \]
  Assign every remaining domain unit vector to a distinct target
  unit vector with coordinate index greater than $d$. All assigned
  images have norm one and pairwise disjoint finite supports.
  They therefore define a linear isometric embedding. For
  $p=\infty$, use the coordinatewise formula in
  Appendix~\ref{sec:isometries-ell_p}; it defines the operator on
  all of $\ell_\infty$ and preserves $c_0$.
  The domain index is $k_j$ because the corresponding limiting jump
  is $\|\mu_{k_j}\|_p e_{k_j}$. On
  $(\mathscr D_n\cap\mathscr H_n(\gamma))^c$, including outcomes
  where the finite restrictions have not yet been matched, set
  $U_{n,\gamma}$ to be the identity. Increasing enumeration of
  unused target coordinates makes this a measurable choice.

  For each partial sum with $m$ retained jumps, its distance from
  the corresponding image of $Y^{(s)}$ is at most
  \begin{align*}
    \bigg\|\sum_{j=1}^m V_{nj}
       -\sum_{j=1}^m\|\mu_{k_j}\|_p U(e_{k_j})\bigg\|_p
    &\leq R_n(\gamma)
       +\sum_{j=1}^{N_T}
          \big|\|V_{nj}^\gamma\|_p-\|\mu_{k_j}\|_p\big|\\
    &\leq R_n(\gamma)
       +\sum_{j=1}^{N_T}
          \big|\|V_{nj}^\gamma\|_p-\|\mu_{k_j}^\gamma\|_p\big|
       +\sum_{j=1}^{N_T}
          (\|\mu_{k_j}\|_p-\|\mu_{k_j}^\gamma\|_p)\\
    &\leq R_n(\gamma)+E_n(\gamma)+R(\gamma).
  \end{align*}
  The first step uses the triangle inequality and
  $V_{nj}^\gamma=\|V_{nj}^\gamma\|_p U(e_{k_j})$.
  The last step uses that the norm is $1$-Lipschitz in $\dm_p$ and
  that adding back the omitted atoms changes the norm by at most
  their norm. Thus, this bound also controls the Hausdorff distance
  between the two finite ranges. Maxima are handled in the same way
  after coordinate truncation.

  For fixed $\gamma$, the error bound converges almost surely to
  $2R(\gamma)$. Hence its exceedance probability has limsup at most
  $\P\{2R(\gamma)\geq\eps\}$ for any $\eps>0$. First let
  $n\to\infty$, then let $\gamma\downarrow0$, using the bounds on
  $\sD_n^c$ and $\mathscr{H}_n(\gamma)^c$. This proves the assertions.
  A deterministic diagonal choice of $\gamma=\gamma(n)$ gives the
  sequence of embeddings for the stated coupling formulation.
\end{proof}

\section{Examples}
\label{sec:examples}

\subsection{Regular tail measures}
\label{sec:exampl-regul-tail}

In the following we write $a_n\sim b_n$ if $\lim_{n\to\infty}
a_n/b_n=1$. We also often encounter the measure $\theta_\alpha$ on
$(0,\infty)$ defined by $\theta_\alpha((s,\infty))=s^{-\alpha}$,
$s>0$, with $\alpha>0$. 

\begin{example}
  \label{ex:iid}
  Let $X_1^{(d)} = (\xi_1,\ldots,\xi_d)$, where $\xi_1,\dots,\xi_d$
  are independent copies of a nonnegative random variable $\xi$ such that its tail
  $\overline{F}(t)=\Prob{\xi> t}$ is regularly varying at infinity of
  order $\alpha>0$. Then $\overline{F}(t)=L(t)t^{-\alpha}$, $t>0$, for
  a slowly varying function $L$. We systematically use the Potter
  bound $L(t)\leq ct^\delta$, which is valid for any $\delta>0$ with a
  constant $c>0$ and for all sufficiently large $t$, see
  \cite[Section~B.2]{bmd}. Let
  \begin{equation}
  \label{eq_iid_a_n}
    a_n=\sup \big\{t: \overline{F}(t)>(nd)^{-1}\big\}.
  \end{equation}
  For instance, if $\overline{F}(t)=t^{-\alpha}$ for all sufficiently
  large $t$, then $a_n = (nd)^{1/\alpha}$ for all
  sufficiently large $n$. In general,
  \begin{equation}
    \label{eq:7}
    (nd)^{-1}\sim \overline{F}(a_n)=L(a_n)a_n^{-\alpha},
  \end{equation}
  and $a_n\geq (nd)^{1/\alpha-\delta}$ for all $\delta>0$ and
  all sufficiently large $n$. 
  
  We derive \eqref{eq:nu} for $p=\infty$ from
  Lemma~\ref{lemma:dtw}. Let $f$ be a bounded nonnegative continuous
  function supported by $(s,\infty)$. First,
  \begin{align*}
    n\E
    &\bigg[\exp\Big\{-\sum_{i=1}^d f(a_n^{-1}X_{1i}^{(d)})\Big\}
      \one_{\chi_{1}^{(d)}((a_ns,\infty))\geq 2}\bigg]
    \leq n \Prob{\sum_{i=1}^d \one_{\xi_i>a_ns}\geq 2}\\
    &\leq n \binom{d}{2} \Prob{\xi>a_ns}^{2}
    =n \binom{d}{2} \overline{F}(sa_n)^{2}
    =n\binom{d}{2}\overline{F}(a_n)^{2}
      \bigg(\frac{\overline{F}(sa_n)}{\overline{F}(a_n)}\bigg)^{2}.
  \end{align*}
  The last factor converges to $s^{-2\alpha}$ by the regular
  variation property of $\overline{F}$. The limit of the remaining
  part is zero by the choice of $a_n$.
  Thus, to check \eqref{eq:dtw}, we need to consider
  \begin{align*}
    n\E &\Big[\exp\Big\{-\sum_{i=1}^d f(a_n^{-1}X_{1i}^{(d)})\Big\}
    \one_{\chi_{1}^{(d)}((a_ns,\infty))=1}\Big]\\
    &=nd
    \E \Big[\exp\Big\{-\sum_{i=1}^d f(a_n^{-1}\xi_i)\Big\}
    \one_{\xi_1>a_ns}
    \one_{\xi_{2}\leq a_ns,\dots,\xi_d\leq a_ns}\Big]\\
    &=nd
    \E \Big[e^{-f(a_n^{-1}\xi)}
      \one_{\xi>a_ns}\Big]\Prob{\xi\leq a_ns}^{d-1},
  \end{align*}
  where we have used that $f(x)=0$ if $x\leq s$. Note that
  \begin{displaymath}
    \lim_{n\to\infty} \Prob{\xi\leq a_ns}^{d-1}
    =\exp\Big\{-\lim_{n\to\infty} (d-1)\overline{F}(a_ns)\Big\}=1.
  \end{displaymath}
  Denote $g(x)=e^{-f(x)}\one_{x>s}$.  Since $\xi$ is regularly varying
  with the tail measure $\theta_\alpha$ and $g$ is supported by
  $(s,\infty)$ and is discontinuous at one point of tail measure zero,
  \begin{displaymath}
    nd\E \Big[e^{-f(a_n^{-1}\xi)}\one_{\xi>a_ns}\Big]
    =nd\E g(a_n^{-1}\xi)\to \int g\d\theta_\alpha
    \quad\text{as}\; n\to\infty.
  \end{displaymath}
  Thus,
  \begin{multline*}
    n\E \Big[\exp\Big\{-\sum f(a_n^{-1}X_{1i}^{(d)})\Big\}
    \one_{\chi_{1}^{(d)}((a_ns,\infty))=1}\Big]
    \to \int_0^\infty
    e^{-f(x)}\one_{x>s}\theta_\alpha(\d x)\\
    =\int e^{-\int f \d\mu}\one_{\mu((s,\infty))\geq 1}
    \nu(\d\mu)\quad\text{as}\; n\to\infty,
  \end{multline*}
  so that \eqref{eq:nu} holds with $\nu$ being the pushforward of
  $\theta_\alpha$ under the map $x\mapsto\delta_x$.
  
  Condition (B) holds by Lemma~\ref{lemma:exchangeable-overlap}.
  By
  Theorem~\ref{th:single-jump}, the random metric spaces built of
  partial maxima with the $\ell_{\infty}$ metric, after normalisation,
  converge in distribution to the space whose metric is induced by the
  crinkled subordinator, generated by the measure $\nu$ supported by
  Dirac measures as specified above.
  
  Now consider the $\ell_p$ case with $p\in[1,\infty)$ and assuming
  that $p>\alpha$. First, we check \eqref{eq:nu} by applying
  Lemma~\ref{lemma:lp-vague-total}. Define $(b_d)_{d\in\NN}$ from the
  requirement that $\overline{F}(b_d)\sim d^{-1}$. Then,
  $b_d^{-p}\sum_{i=1}^d \xi_i^p$ converges in distribution to a
  one-sided $(\alpha/p)$-stable distribution, see
  \cite[Theorem~16.25]{MR1422917}.
  Moreover, $a_n/b_{d(n)}\to\infty$:
  otherwise, along a subsequence $a_n\leq Cb_{d(n)}$ for some fixed
  $C$, and regular variation would contradict
  $\overline F(a_n)/\overline F(b_{d(n)})\sim n^{-1}\to0$.
  Thus $x_n=(a_n/b_{d(n)})^p\to\infty$.
  We use the large-deviation form of \cite{MR240854}:
  for nonnegative i.i.d.\ $W_i$ with tail index $\beta\in(0,1)$,
  $\Prob{\sum_{i=1}^m W_i>t}\sim m\Prob{W_1>t}$ whenever
  $t\to\infty$ and $m\Prob{W_1>t}\to0$. This result applies along
  arbitrary triangular sequences. Therefore,
  \begin{align*}
    n\Prob{\|X_1^{(d)}\|_p>sa_n}
    &=n\Prob{\xi_1^p+\cdots+\xi_d^p>s^p b_d^p x_n}\\
    &\sim nd\Prob{\xi_1^p>s^p b_d^p x_n}=nd\Prob{\xi_1>sa_n}\to
      s^{-\alpha} \quad \text{as}\; n\to\infty. 
  \end{align*}
  Thus, \eqref{eq:p-norm-convergence} holds and
  Lemma~\ref{lemma:lp-vague-total} applies. Furthermore,
  \eqref{eq:levy-integral} holds, since
  \begin{displaymath}
    \int \min(1,\|\mu\|_p^p)\nu(\d \mu)
    =\int_{0}^{1}x^p\alpha  x^{-\alpha-1}\d x
    +\int_{1}^{\infty}\alpha x^{-\alpha-1}\d  x
    =\frac{\alpha}{p-\alpha}+1<\infty.
  \end{displaymath}
  
  Karamata's theorem \cite[Section~B.4]{bmd}, applied to the regularly
  varying tail of $\xi^p$ with index $\alpha/p$, yields
  \begin{equation}
    \label{eq:con_last_iid}
    \lim_{x\to\infty}\frac{x\Prob{\xi^p>x}}
    {\E\big(\xi^p\one_{\xi^p\leq x}\big)}
    =\frac{1-\alpha/p}{\alpha/p}=\frac{p-\alpha}{\alpha}.
  \end{equation}
  Hence, 
  \begin{align*}
    \frac{n}{a_n^p}
      \E\big(\|X_1^{(d)}\|_p^p\one_{\|X_1^{(d)}\|_p\leq sa_n}\big)
    \leq \frac{nd}{a_n^p}\E\big(\xi^p\one_{\xi\leq
      sa_n}\big)
    \sim
      \frac{nd}{a_n^p}\frac{\alpha}{p-\alpha}(sa_n)^{p}
      \Prob{\xi>sa_n},
  \end{align*}
  so that \eqref{eq:b1} follows from 
  \begin{align*}
    \lim_{s\downarrow 0}\limsup_{n\to\infty}\frac{n}{a_n^p}
    \E\big(\|X_1^{(d)}\|_p^p\one_{\|X_1^{(d)}\|_p\leq sa_n}\big)
    &\leq \lim_{s\downarrow 0}\frac{\alpha}{p-\alpha}s^p
    \limsup_{n\to\infty}nd\Prob{\xi>sa_n}\\
    &=\lim_{s\downarrow 0}\frac{\alpha}{p-\alpha}s^{p-\alpha}=0.
  \end{align*}
  Condition \eqref{eq:b} for the random walk setting holds if
  $\alpha\in(0,1)$, since \eqref{eq:con_last_iid} with $p=1$ yields
  that 
  \begin{displaymath}
    \frac{n}{a_n}
      \E\big(\|X_1^{(d)}\|_p\one_{\|X_1^{(d)}\|_p\leq sa_n}\big)
    \leq \frac{nd}{a_n}\E\big(\xi\one_{\xi\leq sa_n}\big)
    \sim\frac{nd}{a_n}\frac{\alpha}{1-\alpha}(sa_n)\Prob{\xi>sa_n},
  \end{displaymath}
  and
  \begin{align*}
    \lim_{s\downarrow0}\limsup_{n\to\infty}\frac{n}{a_n}
    \E\big(\|X_1^{(d)}\|_p\one_{\|X_1^{(d)}\|_p\leq sa_n}\big)
    &\leq \lim_{s\downarrow 0}\frac{\alpha}{1-\alpha}s
      \limsup_{n\to\infty}nd\Prob{\xi>sa_n}\\
    &=\lim_{s\downarrow 0}\frac{\alpha}{1-\alpha}s^{1-\alpha}=0.
  \end{align*}
  Thus, Theorem~\ref{th:single-jump} holds in the maximum scheme with
  $p\in[1,\infty)$ under the assumption that $p>\alpha$. In the random
  walk scheme, the result holds for all $p\in[1,\infty]$, assuming
  that $\alpha\in(0,1)$.
\end{example}

\begin{example}
  \label{exa_4_7}
  Let $X_1^{(d)} = \zeta e_{\eta_d}$, where $e_{\eta_d}$ is the
  $\eta_d$th basis vector and $\eta_d\in \{1,\dots,d\}$ is a random
  variable with distribution $\Prob{\eta_d = k} = p_{kd}$,
  $k=1,\dots,d$, and such that
  \begin{equation}
    \label{eq:one_con}
    \sum_{k = 1}^{d} p_{kd}^2 \to 0 \quad \text{as}\ d \to \infty.
  \end{equation}
  The latter holds, e.g., if $\max_{1\leq k\leq d}p_{kd}\to 0$, which
  is the case if $\eta_d$ equally likely takes all possible values.
  Furthermore, assume that $\zeta$ is a nonnegative random variable
  with regularly varying tail $\overline{F}$ of order $\alpha$,
  independent of $\eta_d$. Let
  \begin{equation}
    \label{eq:a_n}
    a_n=\sup\{t:\overline{F}(t)>n^{-1}\}.
  \end{equation}
  
  Note that $\chi_1^{(d)}=\delta_\zeta$.
  If $f$ is a bounded continuous function supported by
  $(s,\infty)$, then
  \begin{align*}
    n\E&\Big[\exp\Big\{-\sum_{i=1}^{d}f(a_n^{-1}X_{1i}^{(d)})\Big\}
    \one_{\chi_1^{(d)}((a_ns,\infty))\geq 1}\Big]\\
       &=n\E\Big[e^{-f(a_n^{-1}\zeta)}\one_{\zeta>a_ns}\Big]
         \to \int e^{-\int f\d\mu}\one_{\mu((s,\infty))=1}\nu(\d\mu)
         \quad \text{as}\; n\to\infty,
  \end{align*}
  where $\nu$ is the pushforward of $\theta_\alpha$ under the map
  $x\mapsto\delta_x$.
  Since all involved counting measures are Dirac measures,
  \eqref{eq:nu} holds for all $p\in[1,\infty]$, and
  for finite $p$, \eqref{eq:levy-integral} holds when $p>\alpha$,
  by the same calculation as in Example~\ref{ex:iid}. Furthermore, (B) holds, since by
  \eqref{eq:one_con}
  \begin{align*}
    n^2 \Prob{\| X_1^{(d)} \wedge X_2^{(d)} \|_{\infty} > \eps a_n}
      & = n^2 \sum_{k = 1}^{d} \big(\Prob{\zeta >\eps a_n}\big)^2
        \big(\Prob{\eta_d = k}\big)^2 \\
      & = n^2 \big(\Prob{\zeta >\eps a_n}\big)^2 \sum_{k = 1}^{d}
        p_{kd}^2
        \to 0\quad \text{as}\; n\to\infty.
  \end{align*}
  Finally, assuming that $\alpha<p$, \eqref{eq:con_last_iid} implies that
  \begin{displaymath}
    \frac{n}{a_n^p}
    \E\big(\|X_1^{(d)}\|_p^p\one_{\|X_1^{(d)}\|_p\leq sa_n}\big)
    =\frac{n}{a_n^p}
    \E\big(\zeta^p\one_{\zeta\leq sa_n}\big)
    \sim\frac{\alpha}{p-\alpha}s^p n\Prob{\zeta>sa_n},
  \end{displaymath}
  and
  \begin{align*}
    \lim_{s\downarrow 0}\limsup_{n\to\infty}\frac{n}{a_n^p}
    \E\big(\|X_1^{(d)}\|_p^p\one_{\|X_1^{(d)}\|_p\leq sa_n}\big) 
    &=\lim_{s\downarrow 0}\frac{\alpha}{p-\alpha}s^p
      \limsup_{n\to\infty}n\Prob{\zeta>sa_n}\\
    &=\lim_{s\downarrow 0}\frac{\alpha}{p-\alpha}s^{p-\alpha}=0.
  \end{align*}
  Condition \eqref{eq:b} holds if $\alpha\in(0,1)$.  Therefore, the
  statements of Theorem~\ref{th:single-jump} hold in the maximum
  scheme, for any $p\in[1,\infty]$ with $p>\alpha$, and in the random
  walk scheme if $\alpha<1$.
\end{example}

\begin{example}
  \label{ex:dependent-location}
  We alter the setting of Example~\ref{exa_4_7} by assuming that the
  value of a single nontrivial component of $X_1^{(d)}$ depends on its
  position.  Let $X_{1i}^{(d)} = \zeta \one_{i\leq\zeta<i+1}$ for
    $i=1,\dots,d$, where
  $\zeta$ is a nonnegative random variable with regularly
  varying tail of index $\alpha\in(0,p)$. Define $a_n$ by
  \eqref{eq:a_n}.  For any bounded continuous function
  $f$ with support in $(s,\infty)$, we have
  \begin{align*}
    n\E\bigg[\exp\Big\{-\sum_{i=1}^{d}
    &f(a_n^{-1}X_{1i}^{(d)})\Big\}
    \one_{\chi_1^{(d)}((a_ns,\infty))\geq 1}\bigg]\\
    &=n\E\bigg[e^{-f(a_n^{-1}\zeta)}
    \one_{a_n^{-1}\zeta>s}\one_{1\leq
      \zeta< d+1}\bigg]\\
    &\to \int e^{-\int f(u)d\mu}\one_{\mu((s,\infty))=1}\nu(\d\mu),
  \end{align*}
  if $\lim_{n\to\infty}a_n^{-1}d=\infty$. The latter fact holds if
  $n=o(d^{\alpha-\delta})$ for some $\delta>0$. 
  The tail measure $\nu$ is the pushforward of 
  $\theta_\alpha$ under the map $x\mapsto\delta_x$.

  Note that every counting measure involved has at most one nonzero
  atom.  On this subspace all the metrics $\dm_p$ coincide, so the
  already established convergence \eqref{eq:nu} for $p=\infty$ holds
  for every $p\in[1,\infty]$. Moreover, for finite $p$,
  \eqref{eq:levy-integral} holds when $p>\alpha$, by the same
  calculation as in Example~\ref{ex:iid}.

  Denote $p_i=\Prob{i\leq\zeta<i+1}$ and
  $\bar F(i)=\Prob{\zeta\geq i}$. Then
  \begin{align*}
    n^2 
    &\Prob{\| X_1^{(d)} \wedge X_2^{(d)} \|_{\infty} >\eps a_n}
      = n^2 \Prob{\bigcup_{i = 1}^{d} 
      \{ X_{1i}^{(d)} \wedge X_{2i}^{(d)} >\eps a_n\}} \\
    &= n^2 \sum_{i = 1}^{d} 
      \big(\Prob{\zeta >\eps a_n, i\leq \zeta<i+1}\big)^2
      \leq n^2\sum_{i = \lfloor\eps a_n\rfloor}^{\infty}
      p_i^2\\
    &\leq \sup_{i\geq \lfloor\eps a_n\rfloor}
      \frac{p_i}{\bar F(i)}\; \Big(n^2\bar F(\lfloor\eps a_n\rfloor)^2\Big),
  \end{align*}
  where $\lfloor\eps a_n\rfloor$ denotes the integer part of $\eps
  a_n$.  The first factor converges to zero and the second one to a
  constant by the regular variation assumption on $\zeta$.
  Thus, (B) holds.  Finally, (C) can be verified by following the same
  arguments as in Example~\ref{exa_4_7}.  Thus,
  Theorem~\ref{th:single-jump} holds under our assumptions with any
  $p\in[1,\infty]$ and $\alpha<p$, assuming that
  $n=o(d^{\alpha-\delta})$ for some $\delta>0$. In the random walk
  scheme we need to assume that $\alpha\in(0,1)$.
\end{example}

\begin{example}[Moving maxima]
  \label{ex:moving-max}
  Let $(Z_i)_{i\in\ZZ}$ be independent copies of a nonnegative random
  variable $Z$ with regularly varying tail
  $\overline F(t)=\Prob{Z>t}$ of index $\alpha>0$. Put
  $X_{1i}^{(d)}=Z_i\vee Z_{i-1}$, $1\leq i\leq d$, and let
  different rows use independent copies of the innovation sequence.
  Define $a_n$ by \eqref{eq_iid_a_n} with this $\overline F$, so that
  $nd\overline F(a_n)\to1$. We allow any $d=d(n)\to\infty$.

  First consider $p=\infty$. Fix $s>0$ and a bounded nonnegative
  continuous $f$ supported in $(s,\infty)$. Let
  $N_Z=\sum_{i=0}^d\one_{Z_i>sa_n}$. Then
  \begin{align*}
    n\Prob{N_Z\geq2}&\leq n\binom{d+1}{2}\overline F(sa_n)^2\to0,\\
    n\Prob{Z_0>sa_n\text{ or }Z_d>sa_n}
       &\leq2n\overline F(sa_n)\to0.
  \end{align*}
  A unique large interior innovation produces exactly two atoms at
  its value. Therefore
  \begin{align*}
    &n\E\left[e^{-\int f\d(a_n^{-1}\chi_1^{(d)})}
                \one_{\chi_1^{(d)}((sa_n,\infty))\geq1}\right]\\
    &\quad=n(d-1)\E\left[e^{-2f(Z/a_n)}\one_{Z>sa_n}\right]
                   (1-\overline F(sa_n))^d+o(1)
    \to\int_{(s,\infty)}e^{-2f(x)}\theta_\alpha(\d x).
  \end{align*}
  By Lemma~\ref{lemma:dtw}, (A) holds for $p=\infty$, with $\nu$
  the pushforward of $\theta_\alpha$ under $x\mapsto2\delta_x$.

  For finite $p>\alpha$, the deterministic bounds
  \begin{displaymath}
    2\max_{1\leq i\leq d-1}Z_i^p
       \leq\|X_1^{(d)}\|_p^p
       \leq2\sum_{i=0}^d Z_i^p
  \end{displaymath}
  give the norm-tail convergence directly. The lower bound holds
  because each interior innovation occurs in two adjacent maxima; the
  upper bound follows from $\max(a,b)^p\leq a^p+b^p$.  The
  probabilities that either bounding variable exceeds $(sa_n)^p$,
  multiplied by $n$, both converge to $2^{\alpha/p}s^{-\alpha}$ . For
  the upper bound, use the large deviation estimate from
  Example~\ref{ex:iid} for the $d+1$ independent variables $Z_i^p$,
  whose tail index is $\alpha/p<1$. Hence
  \begin{displaymath}
    n\Prob{\|X_1^{(d)}\|_p>sa_n}
        \to 2^{\alpha/p}s^{-\alpha}
        =\nu(\{\mu:\|\mu\|_p>s\}).
  \end{displaymath}
  Lemma~\ref{lemma:lp-vague-total} now proves (A) in $\bN_p$.
  No additional growth restriction on $d(n)$ is required. Moreover,
  the exact integrability calculation is
  \begin{align*}
    \int\min(1,\|\mu\|_p^p)\nu(\d\mu)
      &=\int_0^\infty\min(1,2x^p)\theta_\alpha(\d x)\\
      &=2\int_0^{2^{-1/p}}\alpha x^{p-\alpha-1}\d x
          +\int_{2^{-1/p}}^\infty\alpha x^{-\alpha-1}\d x
        =2^{\alpha/p}\frac{p}{p-\alpha}<\infty.
  \end{align*}

  For an independent innovation sequence $(Z_i')$, the union bound
  and independence of the rows yield
  \begin{align*}
    n^2\Prob{\|X_1^{(d)}\wedge X_2^{(d)}\|_\infty>sa_n}
      &\leq n^2d\Prob{(Z_0\vee Z_1)\wedge(Z_0'\vee Z_1')>sa_n}\\
      &\leq4n^2d\overline F(sa_n)^2\to 0.
  \end{align*}
  The limiting supremum-norm tail is positive at every level, so (B)
  follows. For finite $p>\alpha$, Karamata's theorem gives
  \begin{align*}
    \limsup_{n\to\infty}\frac{n}{a_n^p}
    \E\left[\|X_1^{(d)}\|_p^p
    \one_{\|X_1^{(d)}\|_p\leq sa_n}\right]
    \leq \limsup_{n\to\infty} \frac{2n(d+1)}{a_n^p}\E[Z^p\one_{Z\leq sa_n}]
    =\frac{2\alpha}{p-\alpha}s^{p-\alpha}.
  \end{align*}
  Letting $s\downarrow0$ verifies \eqref{eq:b1}. If $\alpha<1$,
  $\|X_1^{(d)}\|_p\leq\|X_1^{(d)}\|_1$ similarly gives, for every
  $p\in[1,\infty]$,
  \begin{displaymath}
    \limsup_{n\to\infty}\frac{n}{a_n}
       \E\left[\|X_1^{(d)}\|_p\one_{\|X_1^{(d)}\|_p\leq sa_n}\right]
       \leq\frac{2\alpha}{1-\alpha}s^{1-\alpha},
  \end{displaymath}
  proving \eqref{eq:b}. Thus, Theorem~\ref{th:single-jump} applies
  to maxima for $p>\alpha$ (including $p=\infty$), and to random
  walks for every $p\in[1,\infty]$ when $\alpha<1$.
\end{example}

\subsection{Vectors sampled from long-range dependent time
  series}
\label{sec:example:-long-range}

Consider the \emph{logistic} max-stable distribution
\begin{equation}
  \label{eq:Kp}
  \Prob{X^{(d)}\leq u}
  =\exp\{-\|u^{-1}\|_r\},\qquad u\in(0,\infty)^d,
\end{equation}
where $r\in[1,\infty]$ and $u^{-1}=(u_1^{-1},\ldots,u_d^{-1})$; see
\cite{MR2790117}.  The cases $r=1$ and $r=\infty$ correspond,
respectively, to independent unit Fr\'echet components and identical
components. We assume $1<r<\infty$ and put $\beta=1/r$.

Let $\zeta$ be a positive $\beta$-stable random variable with
$\E e^{-t\zeta}=e^{-t^\beta}$, and let $(\eta_i)_{i\in\ZZ}$ be
i.i.d.\ random variables, independent of $\zeta$, with
$\Prob{\eta_i\leq x}=e^{-x^{-r}}$, $x>0$. Then
\[
  X^{(d)}\dsim\zeta^{1/r}(\eta_1,\ldots,\eta_d),
\]
since the conditional joint distribution is
$\exp\{-\zeta\sum_{i=1}^d u_i^{-r}\}$. This representation gives a
stationary exchangeable sequence $Z_i=\zeta^{1/r}\eta_i$ whose
finite-dimensional distributions are \eqref{eq:Kp}. Its dependence
does not decay with the lag, and the sequence is not mixing.

Let $d=d(n)\to\infty$ and $a_n=nd^{1/r}$. We prove (A) with $p=\infty$
by identifying its tail measure explicitly. Write
$\theta_r(\d x)=r x^{-r-1}\d x$, $x>0$, and
$c_\beta=\beta/\Gamma(1-\beta)$.  Furthermore, let $\Pi_\lambda$ be a
Poisson point process on $(0,\infty)$ with intensity
$\lambda\theta_r$, for $\lambda>0$. Define, for Borel
$D\subset\bN_\infty\setminus\{0\}$,
\begin{equation}
  \label{eq:logistic-mixture}
  \nu(D)=c_\beta\int_0^\infty
  \Prob{\Pi_\lambda\in D}\lambda^{-1-\beta}\d\lambda.
\end{equation}
For a bounded nonnegative Borel function $f$ vanishing on a
neighbourhood of zero, the Laplace functional of $\Pi_\lambda$ is
\begin{displaymath}
  \E \exp\left\{-\int f\d\Pi_\lambda\right\}
  =e^{-\lambda A_{f,0}},
\end{displaymath}
where
\begin{displaymath}
  A_{f,s}=\int_s^\infty \big(1-e^{-f(x)}\big)\theta_r(\d x),\quad s\geq 0. 
\end{displaymath}
For $a>0$, integration by parts gives
\begin{displaymath}
  c_\beta\int_0^\infty(1-e^{-a\lambda})\lambda^{-1-\beta}\d\lambda
  =\frac{c_\beta a}{\beta}\int_0^\infty
       e^{-a\lambda}\lambda^{-\beta}\d\lambda
  =a^\beta.
\end{displaymath}
The boundary terms vanish because $0<\beta<1$, and the identity is
also valid for $a=0$. Thus, by Tonelli's theorem,
\begin{multline}
  \label{eq:logistic-laplace}
  \int\big(1-e^{-\int f\d\mu}\big)\nu(\d\mu)
  =c_\beta \int_0^\infty\left[1-\E e^{-\int f\d \Pi_\lambda}\right]
  \lambda^{-1-\beta}\d\lambda\\
  =c_\beta \int_0^\infty (1-e^{-\lambda
    A_{f,0}})\lambda^{-1-\beta}\d\lambda
  =A_{f,0}^\beta.
\end{multline}
Since $\theta_r((s,\infty))=s^{-r}$, the count
$\Pi_\lambda((s,\infty))$ is Poisson with mean $\lambda s^{-r}$.
Therefore,
\[
  \nu(\{\mu:\|\mu\|_\infty>s\})
  =c_\beta\int_0^\infty(1-e^{-\lambda s^{-r}})
  \lambda^{-1-\beta}\d\lambda=s^{-1},\qquad s>0.
\]
Thus $\nu$ is finite on $\sS_\infty$, every $s>0$ is a continuity
level for its supremum-norm tail, and $\overline\nu_\infty=\theta_1$.

Conditional on
$\zeta/n^r=\lambda$, the normalised counting measure
$a_n^{-1}\chi^{(d)}$ has the distribution of
\[
  Q_{d,\lambda}=\sum_{i=1}^d
       \delta_{(\lambda/d)^{1/r}\eta_i}.
\]
Fix $s>0$ and a bounded nonnegative continuous function $f$ on
$[0,\infty)$ such that $f=0$ on $[0,s]$. Then
$A_{f,0}=A_{f,s}$ and $0\leq A_{f,s}\leq s^{-r}$.
Each atom of $Q_{d,\lambda}$ has distribution
$(\lambda/d)e^{-(\lambda/d)x^{-r}}\theta_r(\d x)$. Therefore
\begin{equation}
  \label{eq:Hd-conv}
  H_d(\lambda)=\E e^{-\int f\d Q_{d,\lambda}}
  =\left[1-\frac{\lambda}{d}J_d(\lambda)\right]^d
  \to e^{-\lambda A_{f,s}},
\end{equation}
where
\begin{displaymath}
  J_d(\lambda)=\int_{(s,\infty)}(1-e^{-f(x)})
  e^{-(\lambda/d)x^{-r}}\theta_r(\d x).
\end{displaymath}
To justify the convergence uniformly for $0\leq\lambda\leq M$,
use $1-e^{-u}\leq u$, $u\geq0$, to obtain
\begin{align*}
  0\leq A_{f,s}-J_d(\lambda)
  &=\int_{(s,\infty)}(1-e^{-f(x)})
       (1-e^{-(\lambda/d)x^{-r}})\theta_r(\d x)\\
  &\leq\frac{\lambda}{d}\int_{(s,\infty)}
       (1-e^{-f(x)})x^{-r}\theta_r(\d x)\\
  &\leq d^{-1} Ms^{-r}A_{f,s}
  \leq  d^{-1} Ms^{-2r}.
\end{align*}
Thus $J_d(\lambda)\uparrow A_{f,s}$ pointwise, and the convergence is
uniform on $[0,M]$. Put $u_d(\lambda)=\lambda J_d(\lambda)/d$.  For
$d\geq2Ms^{-r}$, we have $0\leq u_d(\lambda)\leq1/2$ for
$\lambda\in[0,M]$. The inequality $|\log(1-u)+u|\leq u^2$,
$0\leq u\leq1/2$, now gives
\begin{align*}
  |\log H_d(\lambda)+\lambda A_{f,s}|
  \leq d|\log(1-u_d(\lambda))+u_d(\lambda)|
       +\lambda|A_{f,s}-J_d(\lambda)|
  \leq 2d^{-1} M^2s^{-2r}.
\end{align*}
Both exponents are nonpositive, and the exponential function is
$1$-Lipschitz on $(-\infty,0]$. Consequently,
\begin{displaymath}
  \sup_{0\leq\lambda\leq M}
  |H_d(\lambda)-e^{-\lambda A_{f,s}}|
  \leq 2d^{-1}M^2s^{-2r},
\end{displaymath}
which proves locally uniform convergence in \eqref{eq:Hd-conv}.

Independence of the atoms gives
\begin{displaymath}
  \Prob{Q_{d,\lambda}((s,\infty))=0}
  =\big(e^{-\lambda s^{-r}/d}\big)^d=e^{-\lambda s^{-r}}.
\end{displaymath}
On this event $\int f\d Q_{d,\lambda}=0$. Hence
\begin{align*}
  G_d(\lambda)
  &=\E\left[e^{-\int f\d Q_{d,\lambda}}
       \one_{\{Q_{d,\lambda}((s,\infty))>0\}}\right]\\
  &=H_d(\lambda)-e^{-\lambda s^{-r}}
  \to g(\lambda)=e^{-\lambda A_{f,s}}-e^{-\lambda s^{-r}}
\end{align*}
locally uniformly, and
\begin{displaymath}
  0\leq G_d(\lambda)\leq1-e^{-\lambda s^{-r}}
       \leq\min(1,s^{-r}\lambda).
\end{displaymath}

To verify convergence to \eqref{eq:logistic-mixture}, set
$\tau_n(\d\lambda)=n\Prob{\zeta/n^r\in\d\lambda}$. The standard
positive-stable tail asymptotic, obtained from the Laplace transform
by Karamata's Tauberian theorem, is
\[
  \Prob{\zeta>x}\sim\frac{x^{-\beta}}{\Gamma(1-\beta)};
\]
see \cite{MR4230105}. Hence $\tau_n$ vaguely converges to
$\tau(\d\lambda)=c_\beta\lambda^{-1-\beta}\d\lambda$ on $(0,\infty)$.
For fixed $0<\delta<M<\infty$, vague convergence gives
$\tau_n([\delta,M])\to\tau([\delta,M])<\infty$, since $\tau$ has no
atoms at the endpoints. Therefore
\begin{multline*}
  \left|\int_{[\delta,M]}G_{d(n)}\d\tau_n
       -\int_{[\delta,M]}g\d\tau\right|\\
  \leq\sup_{\lambda\in[\delta,M]}|G_{d(n)}(\lambda)-g(\lambda)|
       \,\tau_n([\delta,M])
  +\left|\int_{[\delta,M]}g\d\tau_n
       -\int_{[\delta,M]}g\d\tau\right|\to 0.
\end{multline*}
The first term vanishes by uniform convergence and boundedness of
the masses. The second vanishes by vague convergence, continuity
of $g$, and the absence of endpoint atoms.
At infinity, the stable tail asymptotic gives
\begin{displaymath}
  \int_{(M,\infty)}G_{d(n)}\d\tau_n
  \leq \tau_n((M,\infty))
  =n\Prob{\zeta>Mn^r}
  \to \frac{M^{-\beta}}{\Gamma(1-\beta)}
  \quad\text{as}\; n\to\infty,
\end{displaymath}
which can be made arbitrarily small by the choice of large $M$.
Near zero, $G_d(\lambda)\leq s^{-r}\lambda$, so
\begin{displaymath}
  \int_{(0,\delta]}G_{d(n)}\d\tau_n
  \leq s^{-r}\int_{(0,\delta]} \lambda \tau_n(\d\lambda)
  =  s^{-r}n^{1-r}
  \E\big[\zeta\one_{\{\zeta\leq\delta n^r\}}\big].
\end{displaymath}
The truncated-moment form of Karamata's theorem
\cite[Section~B.4]{bmd} yields
\begin{displaymath}
  \E\big[\zeta\one_{\{\zeta\leq x\}}\big]
  \sim\frac{\beta}{(1-\beta)\Gamma(1-\beta)}x^{1-\beta}
  \quad\text{as}\; x\to\infty.
\end{displaymath}
Since $r\beta=1$, it follows that
\begin{displaymath}
  \limsup_{n\to\infty}\int_{(0,\delta]}G_{d(n)}\d\tau_n
  \leq\frac{s^{-r}c_\beta}{1-\beta}\delta^{1-\beta}
  \to 0\quad\text{as}\; \delta\downarrow0.
\end{displaymath}
The limit $g$ satisfies the same bound $\min(1,s^{-r}\lambda)$,
so its integrals against $\tau$ over the omitted ends also vanish.
Combining these estimates proves
\begin{displaymath}
  \int_0^\infty G_{d(n)}\d\tau_n
  \to c_\beta\int_0^\infty
       (e^{-\lambda A_{f,s}}-e^{-\lambda s^{-r}})
       \lambda^{-1-\beta}\d\lambda
  =s^{-1}-A_{f,s}^\beta.
\end{displaymath}
For the last equality, write the integrand as
$(1-e^{-\lambda s^{-r}})-(1-e^{-\lambda A_{f,s}})$ and apply
the scalar integral identity used in \eqref{eq:logistic-laplace}.
Conditioning on $\zeta/n^r$ identifies the expression on the left as
\begin{displaymath}
  \int_0^\infty G_{d(n)}\d\tau_n
  =n\E\left[e^{-\int f\d(a_n^{-1}\chi^{(d)})}
       \one_{\{\chi^{(d)}((a_ns,\infty))>0\}}\right].
\end{displaymath}
The Poisson mixture similarly gives
\begin{align*}
  \int e^{-\int f\d\mu}
       \one_{\{\mu((s,\infty))>0\}}\nu(\d\mu)
  &=c_\beta\int_0^\infty
       (e^{-\lambda A_{f,s}}-e^{-\lambda s^{-r}})
       \lambda^{-1-\beta}\d\lambda\\
  &=s^{-1}-A_{f,s}^\beta.
\end{align*}
All counting measures here are supported on $(0,\infty)$, and every
$s>0$ is a continuity level of the supremum-norm tail. Thus
Lemma~\ref{lemma:dtw} proves (A) with $p=\infty$.

The mixture also identifies the cluster counts directly. Let
$A_1,\ldots,A_m$ be pairwise disjoint Borel subsets of $(0,\infty)$
bounded away from zero, and put
$a_i=\theta_r(A_i)$, $A=\sum_{i=1}^m a_i>0$, and $q_i=a_i/A$.
For integers $n_i\geq0$ with $N=\sum_{i=1}^m n_i\geq1$, independent
Poisson counts and \eqref{eq:logistic-mixture} give
\begin{equation}
  \label{eq:logistic-counts}
  \begin{aligned}
  p_{n_1,\ldots,n_m}
  &=\nu(\{\mu:\mu(A_i)=n_i,\ 1\leq i\leq m\})\\
  &=c_\beta\int_0^\infty e^{-A\lambda}
       \lambda^{N-1-\beta}\d\lambda
       \prod_{i=1}^m\frac{a_i^{n_i}}{n_i!}\\
  &=\frac{\beta\Gamma(N-\beta)}{\Gamma(1-\beta)}
       A^{\beta-N}\prod_{i=1}^m\frac{a_i^{n_i}}{n_i!},
  \end{aligned}
\end{equation}
where $0^0=1$. Write $K=\mu(A_1\cup\cdots\cup A_m)$ for the total
count and let $H=\{\mu:K>0\}$. Its mass is $\nu(H)=A^\beta$, so
$\nu_A(\cdot)=A^{-\beta}\nu(\cdot\cap H)$ is a probability
measure. Under $\nu_A$, the count vector has probability generating
function
\begin{equation}
  \label{eq:logistic-sibuya}
  \sum_{n_1+\cdots+n_m\geq1}
       A^{-\beta}p_{n_1,\ldots,n_m}\prod_{i=1}^m z_i^{n_i}
  =1-\left(\sum_{i=1}^m q_i(1-z_i)\right)^\beta,
  \quad z_i\in[0,1].
\end{equation}
This is the multivariate Sibuya distribution with parameters
$\beta$ and $(q_1,\ldots,q_m)$; see
\cite[Definition~26 and Eq.~(41)]{MR2817615}. More explicitly, the
total count $K$ satisfies
\begin{equation}
  \label{eq:logistic-sibuya-univariate}
  \nu_A\{K=k\}
  =\frac{\beta\Gamma(k-\beta)}{\Gamma(1-\beta)k!},
  \qquad k\geq1,
\end{equation}
and, conditional on $K=k$, the counts have the multinomial law
with $k$ trials and probabilities $q_1,\ldots,q_m$. Individual counts
may be zero. In particular, for a Borel set $D\subset(0,\infty)$ bounded away from
zero,
\[
  \nu(\{\mu:\mu(D)=k\})
  =\theta_r(D)^\beta
  \frac{\beta\Gamma(k-\beta)}{\Gamma(1-\beta)k!},
  \qquad k\geq1.
\]

The consistency relation
\begin{equation}
  \label{eq:p-consistency}
  p_{n_1,\ldots,n_m}
  =\sum_{k=0}^\infty p_{n_1,\ldots,n_m,k},
  \qquad n_1+\cdots+n_m\geq1,
\end{equation}
for an additional disjoint Borel set $A_{m+1}$ bounded away from
zero follows from countable additivity of $\nu$. To make the
generating-function calculation explicit, put
$B=\theta_r(A_{m+1})$. The unnormalised generating function for
nonzero $(m+1)$-tuples is
\[
  F(z_1,\ldots,z_m,z)
  =(A+B)^\beta-
  \left(\sum_{i=1}^m a_i(1-z_i)+B(1-z)\right)^\beta.
\]
Setting the first $m$ variables to zero gives
\[
  F(0,\ldots,0,z)
  =\sum_{k\geq1}p_{0,\ldots,0,k}z^k
  =(A+B)^\beta-(A+B(1-z))^\beta.
\]
Thus, on setting the last variable to one and subtracting,
\[
  F(z_1,\ldots,z_m,1)-F(0,\ldots,0,1)
  =A^\beta-\left(\sum_{i=1}^m a_i(1-z_i)\right)^\beta.
\]
Comparison with $A^\beta$ times \eqref{eq:logistic-sibuya} proves
\eqref{eq:p-consistency} by equality of coefficients. Nonnegativity
justifies the summations and evaluation at one.

Each $\Pi_\lambda$ has infinitely many atoms, but only finitely
many outside any neighbourhood of zero. Hence the same holds for
$\nu$-almost every cluster. For $1\leq p<\infty$, the Poisson
integrability criterion gives 
\[
  \int_{(0,1]}x^p\Pi_\lambda(\d x)<\infty\quad\text{a.s.}
  \quad\Longleftrightarrow\quad
  \int_{(0,1]}x^p\theta_r(\d x)<\infty
  \quad\Longleftrightarrow\quad p>r.
\]
The finitely many larger atoms do not affect this criterion.
Consequently, $\nu(\bN_p^c)=0$ if $p>r$ and $\nu(\bN_p)=0$ if
$p\leq r$.  Since $\dm_\infty\leq\dm_p$, convergence in (A) for finite
$p$ implies the corresponding convergence on $\bN_\infty$ under the
natural inclusion: each bounded $\dm_\infty$-continuous test function
vanishing near $0$ restricts to an admissible test function on
$\bN_p$. Uniqueness of the $p=\infty$ limit would then force $\nu$ to
be concentrated on $\bN_p$. Thus (A) at the normalisation
$a_n=nd^{1/r}$ is impossible for finite $p\leq r$. For $p>r$, the
support calculation alone does not establish finite-$p$ convergence,
which is not asserted here.

Finally, exchangeability and Lemma~\ref{lemma:exchangeable-overlap}
give (B) with $p=\infty$. The case $r=1$ is covered by
Example~\ref{ex:iid}. Thus Theorem~\ref{th:single-jump} applies to the
maximum scheme with $p=\infty$ for every finite $r\geq1$ and every
$d(n)\to\infty$, without a restriction on the relative growth of
$d(n)$ and $n$. No condition (C) is needed for these maxima; a
random-walk limit would require its random-walk part and is not
asserted in this subsection.

\begin{appendices}
  \renewcommand{\theequation}{\thesection.\arabic{equation}}
  
  \setcounter{equation}{0}
  \setcounter{theorem}{0}
  \renewcommand{\thesection}{A}

\appendix

\section{Appendix \thesection: Permutation-invariant metric on \texorpdfstring{$\ell_p$}{ell-p}}
\label{sec:a:-perm-invar}

Let $\Pi$ be the family of finite permutations $\pi:\N\to\N$. For
$\pi\in\Pi$ and $x=(x_1,x_2,\ldots)\in\ell_p$, denote
$\pi(x)=(x_{\pi(1)}, x_{\pi(2)},\ldots)$. Furthermore, define
\begin{equation}
  \label{eq:rhoPi}
  \rho_{\Pi,p}(x,y)=\inf_{\pi\in\Pi}\|\pi(x)-y\|_p,
\end{equation}
which is a permutation-invariant pseudometric on $\ell_p$. It becomes
a metric on the quotient by zero pseudodistance, equivalently by
equality of nonzero coordinate multisets. This identifies more than
just sequences differing by a finite permutation.
In the case
$p=\infty$ we restrict the domain to $c_0$ throughout this appendix.

\begin{lemma}
  \label{le:best_matching_k}
  For $x,y\in\ell_p$ with $p\in[1,\infty)$, and for $x,y\in c_0$
  when $p=\infty$, 
  \begin{displaymath}
    \rho_{\Pi,p}(x,y)
    =\lim_{k\to\infty}\|\tilde{x}_{1:k}-\tilde{y}_{1:k}\|_p
    =\dm_p(\me(x),\me(y)),
  \end{displaymath}
  where $\tilde{x}_{1:k}$ and $\tilde{y}_{1:k}$ are, respectively,
  the finite sequences $(x_1,\dots,x_k)$ and $(y_1,\dots,y_k)$
  rearranged in nondecreasing order. 
\end{lemma}
\begin{proof}
  Note that $\Pi=\cup_{k\in\NN}\Pi_k$, where $\Pi_k$ consists of
  permutations $\pi$ such that $\pi(i)=i$ for all $i>k$. Then
  \begin{displaymath}
    \rho_{\Pi,p}(x,y)=\lim_{k\to\infty}\inf_{\pi\in\Pi_k}\|\pi(x)-y\|_p.
  \end{displaymath}
  Since $\Pi_k\uparrow\Pi$, the sequence of infima is decreasing.  Let
  $x_{1:k}=(x_1,\ldots,x_k)$ and
  $x_{k+1:\infty}=(x_{k+1},x_{k+2},\ldots)$; similarly, define $y_{1:k}$
  and $y_{k+1:\infty}$. For $\pi\in\Pi_k$,
  \begin{displaymath}
    \|\pi(x_{1:k})-y_{1:k}\|_p\leq\|\pi(x)-y\|_p
    \leq \|\pi(x_{1:k})-y_{1:k}\|_p+\|x_{k+1:\infty}-y_{k+1:\infty}\|_p.
  \end{displaymath}
  By Example~2 of \cite{MR1048445} and Section~5.4 of \cite{MR368308},
  \begin{displaymath}
    \inf_{\pi\in\Pi_k}\|\pi(x_{1:k})-y_{1:k}\|_p
    =\|\tilde{x}_{1:k}-\tilde{y}_{1:k}\|_p.
  \end{displaymath}
  Therefore, 
  \begin{align*}
    \|\tilde{x}_{1:k}-\tilde{y}_{1:k}\|_p
    &\leq \inf_{\pi\in\Pi_k}\|\pi(x)-y\|_p
      \leq\inf_{\pi\in\Pi_k}\|\pi(x_{1:k})-y_{1:k}\|_p
      +\|x_{k+1:\infty}-y_{k+1:\infty}\|_p\\
    &=\|\tilde{x}_{1:k}-\tilde{y}_{1:k}\|_p
      +\|x_{k+1:\infty}-y_{k+1:\infty}\|_p.
  \end{align*}
  Letting $k\to\infty$ and noticing that
  $\|x_{k+1:\infty}-y_{k+1:\infty}\|_p\to 0$ confirms
  the claim. The last equality follows from the definition of $\dm_p$
  via optimal matching of atoms.
\end{proof}

\setcounter{theorem}{0}
\setcounter{equation}{0}
\renewcommand{\thesection}{B}

\section{Appendix \thesection: Linear isometric embeddings of \texorpdfstring{$\ell_p$}{ell-p}}
\label{sec:isometries-ell_p}

For finite $p\neq2$, the following characterisation concerns
\emph{linear} isometric embeddings; see \cite{MR105017},
\cite[Theorem~7.4.1]{MR4249569}, and
\cite[Example~5.2.3 and Theorem~5.2.12]{MR1957004}.

\begin{theorem}
  \label{th:isometries}
  Let $1\leq p<\infty$, $p\neq2$, and let
  $U:\ell_p\to\ell_p$ be a bounded linear operator. Then $U$ is an isometric embedding
  if and only if its columns $U(e_k)$ have pairwise disjoint supports
  and $\|U(e_k)\|_p=1$ for every $k$.
\end{theorem}
\begin{proof}
  For an isometric embedding, the norms of $U(e_i)\pm U(e_j)$ equal
  $\|e_i\pm e_j\|_p$. The scalar expression
  $|a+b|^p+|a-b|^p-2|a|^p-2|b|^p$ has one fixed sign for
  $1\leq p<2$ and the opposite sign for $p>2$, and vanishes exactly
  when $ab=0$. Summing it over coordinates therefore forces
  disjointness of $U(e_i)$ and $U(e_j)$. Conversely, disjoint unit
  columns preserve the $p$-norm of every finitely supported vector;
  their norm-continuous extension is a linear isometric embedding.
\end{proof}

Disjoint unit columns also give a sufficient construction when
$p=2$. For $p=\infty$, values on $(e_k)$ alone do not specify an
operator on $\ell_\infty$. The construction needed here is defined
on all bounded sequences as follows. Take pairwise disjoint nonempty
sets $G_k\subset\N$ and numbers $h_{ki}$, $i\in G_k$, such that
$\sup_{i\in G_k}|h_{ki}|=1$, and set
\begin{equation}
  \label{eq:infty-embedding}
  (Ux)_i=
    \begin{cases}h_{ki}x_k,&i\in G_k,\\
                  0,&i\notin\bigcup_kG_k.
    \end{cases}
\end{equation}
This formula defines a linear isometric embedding of $\ell_\infty$
because $\sup_i|(Ux)_i|=\sup_k|x_k|$. When all blocks are finite,
it also maps $c_0$ into $c_0$. This is a sufficient construction,
not a classification of all linear or metric embeddings of
$\ell_\infty$.

\setcounter{equation}{0}
\setcounter{theorem}{0}
\renewcommand{\thesection}{C}

\section{Appendix \thesection: Vague convergence on the counting-measure space}
\label{sec:vague-conv-meas}

Below we present several results concerning vague convergence
on the space of counting measures equipped with the $\dm_p$ metric.
The case of $p=\infty$ relies on considering the ideal $\sS_\infty$,
which was used in \cite{dtw22}. It is easy to show that the weighted
Prokhorov metric used therein defines the same topology as the metric
$\dm_\infty$.  The following result is a reformulation of
\cite[Theorem~2.5]{dtw22} adapted to our notation, see also
\cite[Lemma~4.56]{bmm25}.

\begin{lemma}
  \label{lemma:dtw}
  Condition \eqref{eq:nu} with $p=\infty$ holds if and only if for all
  $s>0$ satisfying $\nu(\{\mu:\|\mu\|_\infty=s\})=0$,
  \begin{multline}
    \label{eq:dtw}
    n \E \Big[e^{-\int f\d(a_n^{-1}\chi^{(d)})}
    \one_{\chi^{(d)}((-\infty,-a_ns)\cup (a_ns,\infty))\geq 1}\Big]\\
    \to \int e^{-\int f \d\mu}
    \one_{\mu((-\infty,-s)\cup(s,\infty))\geq 1}\nu(\!\d\mu)
    \quad \text{as}\; n\to\infty
  \end{multline}
  for each bounded Lipschitz function $f:\R\to\R_+$ supported on
  $(-\infty,-s)\cup (s,\infty)$.
\end{lemma}

\begin{lemma}
  \label{lemma:norm}
  If \eqref{eq:nu} holds, then the measure
  $n\Prob{a_n^{-1}\|\chi^{(d)}\|_p\in \cdot}$ vaguely converges to
  $\overline{\nu}_p$, which is the pushforward of $\nu$ under the map
  $\mu\mapsto\|\mu\|_p$. 
\end{lemma}
\begin{proof}
  The result follows from the continuous mapping theorem, noticing the
  continuity of the function $f(\mu)=\|\mu\|_p$ and that
  $\{\mu:f(\mu)>s\}\in\sS_p$ for all $s>0$. 
\end{proof}

\begin{lemma}
  \label{lemma:pp}
  Condition \eqref{eq:nu} holds if and only if
  \begin{equation}
    \label{eq:pp-chi-n}
    \eta_n=\sum_{i=1}^n\one_{\chi_i^{(d)}\neq0}
                     \delta_{a_n^{-1}\chi_i^{(d)}}
  \end{equation}
  converges in distribution to a Poisson random measure of intensity
  $\nu$ on $\bN_p\setminus\{0\}$. Convergence is in the space of
  point measures finite on every set $\{\mu:\|\mu\|_p>s\}$,
  equipped with the topology induced by bounded continuous test functions
  supported away from $0$. It is not convergence in $\bN_p$ itself.
\end{lemma}
\begin{proof}
  Let $g$ be a nonnegative $\dm_p$-continuous bounded function on $\bN_p$ with
  support in $\sS_p$. The Laplace functional of the point process
  $\eta_n$ from \eqref{eq:pp-chi-n} is given by
  \begin{align*}
    L_n[g]&=\E \exp\Big\{-\sum_{i=1}^n g(a_n^{-1}\chi_i^{(d)})\Big\}\\
    &=\bigg(\E \exp\Big\{-g(a_n^{-1}\chi_i^{(d)})\Big\}\bigg)^n
    =\Big(1-\E h(a_n^{-1}\chi^{(d)})\Big)^n,
  \end{align*}
  where $h(\mu)=1-e^{-g(\mu)}$, $\mu\in\bN_p$.
  Thus, $L_n[g]$ converges
  to $\exp\{-\int(1-e^{-g})\d\nu\}$ if and only if
  $n\E h(a_n^{-1}\chi^{(d)})\to\int h\d\nu$.
  Conversely, any bounded nonnegative continuous test $f$ supported
  away from zero can be written as a constant times $1-e^{-g}$ by
  first scaling $f$ to be strictly smaller than one and taking
  $g=-\log(1-f)$. This yields the stated equivalence.
\end{proof}

The following results make it possible to deduce the vague convergence
on $\sS_p$ for $p\in[1,\infty)$ from the convergence on
$\sS_\infty$. Since these results are useful for their own sake, they
are formulated for a general sequence $(\nu_n)_{n\in\NN}$ of measures
on $(\bN_p,\sN_p)$. While checking \eqref{eq:nu}, we set $n^{-1}\nu_n$ to
be the distribution of $a_n^{-1}\chi^{(d)}$.

\begin{lemma}
  \label{lemma:lp-vague-total}
  Assume that $\nu$ and $\nu_n$, $n\geq1$, are measures on $(\bN_p,\sN_p)$
  which are finite on $\sS_p$ for some $p\in[1,\infty)$. Then
  $\nu_n\vstop \nu$ for $p\in[1,\infty)$ if $\nu_n\vstop[\infty]\nu$
  as $n\to\infty$  and
  \begin{equation}
    \label{eq:p-norm-convergence}
    \lim_{n\to\infty}\nu_n\big(\{\mu:\|\mu\|_p>s\}\big)
    =\nu\big(\{\mu:\|\mu\|_p>s\}\big)<\infty 
  \end{equation}
  for all $s>0$, such that $\nu\big(\{\mu:\|\mu\|_p=s\}\big)=0$.
\end{lemma}
\begin{proof}
  Fix $s>0$ with $\nu(\{\|\mu\|_p=s\})=0$ and write
  $A_s=\{\mu:\|\mu\|_p>s\}$. The map
  $\mu\mapsto\|\mu\|_p$ is lower semicontinuous in $\dm_\infty$:
  its $p$th power is the supremum of the continuous sums over
  finitely many ranked atoms. For a bounded nonnegative
  $\dm_\infty$-continuous function $f$, the Portmanteau inequality,
  applied first with an additional cutoff $\|\mu\|_\infty>\delta$,
  gives
  \begin{displaymath}
    \liminf_n\int_{A_s}f\d\nu_n\geq\int_{A_s}f\d\nu.
  \end{displaymath}
  Here we let continuity levels $\delta\downarrow0$ after applying
  the inequality on sets bounded away from zero in $\dm_\infty$.
  Since \eqref{eq:p-norm-convergence} gives
  $\nu_n(A_s)\to\nu(A_s)<\infty$, applying the same inequality to
  $\|f\|_\infty-f$ proves weak convergence of the finite restrictions
  $\nu_n|_{A_s}$ to $\nu|_{A_s}$ in $\dm_\infty$.

  Write $\alpha_n=\nu_n|_{A_s}$ and $\alpha=\nu|_{A_s}$.  If
  $\alpha(A_s)=0$, then $\alpha_n(A_s)\to0$, which immediately gives
  weak convergence of these finite measures in $\dm_p$.  Suppose
  therefore that $\alpha(A_s)$ is positive.  Let $F(\mu)$ denote the
  pair $(\mu,\|\mu\|_p)$ and write $\lambda_n$ and $\lambda$ for the
  pushforwards of $\alpha_n$ and $\alpha$ under $F$.  These are finite
  measures on $\bN_\infty\times[0,\infty)$, where the first factor
  carries the $\dm_\infty$ topology. Explicitly,
  \[
    \lambda_n(D)=
    \int_{A_s}\one_{\{(\mu,\|\mu\|_p)\in D\}}\,\nu_n(\d\mu)
  \]
  for every Borel set $D$ in this product space, and similarly for
  $\lambda$.  The first marginals of $\lambda_n$ converge weakly to
  $\alpha$, as proved above. Their second marginals converge to the
  pushforward of $\alpha$ under $\mu\mapsto\|\mu\|_p$, by
  \eqref{eq:p-norm-convergence}. Since their total masses also
  converge, the normalised measures $\lambda_n/\alpha_n(A_s)$ are
  tight.

  Let $(\Xi,R)$ have any subsequential limiting law of these
  normalised measures. Lower semicontinuity of the $p$-norm in
  $\dm_\infty$ gives $R\geq\|\Xi\|_p$ almost surely.  The identified
  marginals imply that $R$ and $\|\Xi\|_p$ have the same
  distribution. Applying the bounded strictly increasing function
  $\phi(t)=t/(1+t)$ shows that $R=\|\Xi\|_p$ almost surely. The
  limiting joint distribution is therefore $\lambda/\alpha(A_s)$. This
  identifies every subsequential limit and proves $\lambda_n$ weakly
  converges to $\lambda$ in the product space.

  By a strong representation of the normalised joint convergence, we
  may realise counting measures $\Xi_n,\Xi$ with respective laws
  $\alpha_n/\alpha_n(A_s)$ and $\alpha/\alpha(A_s)$ such that
  $\dm_\infty(\Xi_n,\Xi)\to 0$ and $\|\Xi_n\|_p\to \|\Xi\|_p$ almost
  surely.  Equation~\eqref{eq:topology-norm} then gives
  $\dm_p(\Xi_n,\Xi)\to0$ almost surely. Restoring the total masses
  proves weak convergence of $\alpha_n$ to $\alpha$ in $\dm_p$. Taking
  continuity levels $s\downarrow0$ completes the proof of
  $\nu_n\vstop\nu$.
\end{proof}

Using the ranked sequences from \eqref{eq:matching-metric}, define
\begin{displaymath}
  J_k(\mu)=\sum_{j=1}^k\delta_{u_j(\mu)}
  +\sum_{j=1}^k\delta_{-v_j(\mu)},
\end{displaymath}
with all zero terms discarded. Thus, $J_k$ retains at most $k$ atoms
of each sign and has at most $2k$ atoms. Unlike selecting $k$ atoms
by absolute size with an enumeration-dependent tie rule, this
truncation is continuous in both $\dm_\infty$ and $\dm_p$.
Moreover, $\|J_k\mu\|_p\leq\|\mu\|_p$,
$\dm_p(\mu,J_k\mu)=\|\mu-J_k\mu\|_p$, and $J_k\mu\to\mu$ in
$\dm_p$.

\begin{lemma}
  \label{lemma:lp-vague}
  Assume that $\nu$ and $\nu_n$, $n\geq1$, are measures on
  $(\bN_p,\sN_p)$ which are finite on $\sS_p$ for some
  $p\in[1,\infty)$. Then $\nu_n\vstop \nu$ for $p\in[1,\infty)$ if and
  only if $\nu_n\vstop[\infty]\nu$ and, for all $\eps>0$,
  \begin{equation}
    \label{eq:6}
    \lim_{k\to\infty}\limsup_{n\to\infty}
    \nu_n\big(\{\mu:\|\mu-J_k(\mu)\|_p>\eps\}\big)=0.
  \end{equation}
\end{lemma}
\begin{proof}
  By \cite[Lemma~3.47]{bmm25}, to show that \eqref{eq:vsto}
  holds, it suffices to show convergence of integrals of bounded
  Lipschitz functions. Let $g:\bN_p\to\R$ be a 1-Lipschitz function
  (in the $\dm_p$ metric) whose absolute value is bounded by $C$ and such
  that $g(\mu)=0$ if $\|\mu\|_p\leq r$ for some $r>0$. 

  By the triangle inequality, for any $k\in\N$,
  \begin{multline*}
    \Big|\int g(\mu)\nu_n(\d\mu)-\int g(\mu)\nu(\d\mu)\Big|\leq
    \Big|\int g(\mu)\nu_n(\d\mu)-\int g(\mu)(J_k\nu_n)(\d\mu)\Big|\\
    +\Big|\int g(\mu)(J_k\nu_n)(\d\mu)-\int g(\mu)(J_k\nu)(\d\mu)\Big|
    +\Big|\int g(\mu)(J_k\nu)(\d\mu)-\int g(\mu)\nu(\d\mu)\Big|.
  \end{multline*}
  The second summand converges to zero as $n\to\infty$, since
  $\nu_n\vstop[\infty]\nu$ implies $J_k\nu_n\vstop J_k\nu$, noticing
  that
  $J_k$ is $\dm_\infty$-continuous and the involved counting
  measures have at most $2k$ atoms. On that subspace the
  $\dm_\infty$ and $\dm_p$-topologies coincide, and
  $\|J_k\mu\|_p>r$ implies
  $\|\mu\|_\infty>r/(2k)^{1/p}$. Thus, the composed test functions
  also have support bounded away from zero in $\dm_\infty$.
  For the last summand, note that $J_k\nu \vstop \nu\quad$ as
  $k\to\infty$ by the dominated convergence theorem. 

  The first summand is bounded by
  \begin{displaymath}
    \int
    \big|g(\mu)-g(J_k\mu)\big|\one_{\|\mu-J_k(\mu)\|_p\leq\eps}
    \nu_n(\d\mu)
    +\int \big|g(\mu)-g(J_k\mu)\big|\one_{\|\mu-J_k(\mu)\|_p>\eps}
    \nu_n(\d\mu)
  \end{displaymath}
  for any $\eps>0$. The Lipschitz property of $g$ together with the
  fact that $\dm_p(\mu,J_k(\mu))=\|\mu-J_k(\mu)\|_p$ imply that
  the first term is bounded by
  \begin{displaymath}
    \int_{\{\mu:\|\mu\|_p>r\}} 
    \big|g(\mu)-g(J_k\mu)\big|\one_{\|\mu-J_k(\mu)\|_p\leq\eps}
    \nu_n(\d\mu)
    \leq \eps\nu_n\big(\{\mu:\|\mu\|_p>r\}\big),
  \end{displaymath}
  and so can be made arbitrarily small by the choice of
  $\eps$. The second factor has finite upper limit since
  \begin{displaymath}
    \nu_n\big(\{\|\mu\|_p>r\}\big)
    \leq \nu_n\big(\{\|J_k\mu\|_p>r/2\}\big)
    +\nu_n\big(\{\|\mu-J_k\mu\|_p>r/2\}\big)
  \end{displaymath}
  and then we choose a sufficiently large $k$ such that the second
  term has finite upper limit. The first term has finite upper limit
  since $\|J_k\mu\|_p>r/2$ implies $\|\mu\|_\infty>r/(2(2k)^{1/p})$.

  The second term converges to zero, since
  \begin{align*}
    \int \big|g(\mu)-g(J_k\mu)\big|\one_{\|\mu-J_k(\mu)\|_p>\eps}
    \nu_n(\d\mu)
    \leq  2C\nu_n\big(\{\mu:\|\mu-J_k(\mu)\|_p>\eps\}\big)
  \end{align*}
  in view of \eqref{eq:6}.

  Conversely, suppose $\nu_n\vstop\nu$. To prove \eqref{eq:6},
  restrict to a continuity level below $\eps$. The resulting finite
  measures converge weakly in $\dm_p$ and are uniformly tight. By
  Dini's theorem, the continuous tail functions
  $\mu\mapsto\|\mu-J_k\mu\|_p$ decrease to zero uniformly on each
  compact set. Tightness therefore makes the limsup in \eqref{eq:6}
  arbitrarily small as $k\to\infty$. 
\end{proof}

Since $\dm_\infty\leq\dm_p$ and $\sS_\infty\subset\sS_p$,
every test function for $\sS_\infty$ convergence is also an admissible
test function for $\sS_p$ convergence. Thus, $\nu_n\vstop\nu$ implies
$\nu_n\vstop[\infty]\nu$, with the natural inclusion of state spaces.
In particular, one obtains the following norm-tail consequence.

\begin{lemma}
  \label{lemma:infty-convergence}
  If $\nu_n\vstop\nu$, then 
  \begin{displaymath}
    \nu_n(\{\mu:\|\mu\|_\infty>s\}) \to
    \nu(\{\mu:\|\mu\|_\infty>s\}) \quad\text{as}\; n\to\infty
  \end{displaymath}
  for all $s$ such that $\nu(\{\mu:\|\mu\|_\infty=s\})=0$.
\end{lemma}
\begin{proof}
  The inequality
  $|\|\mu\|_\infty-\|\mu'\|_\infty|
      \leq\dm_\infty(\mu,\mu')\leq\dm_p(\mu,\mu')$
  proves continuity. The inverse image of $(s,\infty)$ is bounded
  away from $0$ in $\dm_p$, and its boundary is null at the stated
  levels. The result follows from the continuous mapping theorem
  and the Portmanteau theorem. 
\end{proof}

\end{appendices}

\backmatter

\bmhead{Acknowledgements}

The authors are grateful to Andrii Ilienko for many helpful
discussions. They also thank the anonymous referees for a very
careful reading of the paper and suggesting numerous improvements.

\end{document}